\documentclass[a4paper,11pt]{jacodesmath-noT1} 

\usepackage{hyperref}
\usepackage[capitalize]{cleveref}

\newcommand{\MR}[1]{\href{http://www.ams.org/mathscinet-getitem?mr=#1}{MR#1}}

\numberwithin{equation}{section}

\theoremstyle{plain}
\newtheorem{prop}[equation]{Proposition}
\newtheorem{cor}[equation]{Corollary}
\newtheorem{lem}[equation]{Lemma}

\crefname{cor}{Corollary}{Corollaries}
\Crefname{cor}{Corollary}{Corollaries}
\crefname{lem}{Lemma}{Lemmas}
\Crefname{lem}{Lemma}{Lemmas}

\theoremstyle{definition}
\newtheorem{notation}[equation]{Notation}
\newtheorem{assumps}[equation]{Assumptions}
\newtheorem{defn}[equation]{Definition}

\theoremstyle{definition}
\newtheorem{rem}[equation]{Remark}

\crefformat{rems}{Remark~#2#1#3}

 \newcounter{case}
 \newenvironment{case}[1][\unskip]{\refstepcounter{case}\bf
 \medskip \noindent Case \thecase\ #1.\normalfont\em\ }{\unskip\upshape}
 \renewcommand{\thecase}{\arabic{case}}
 
 \newcounter{subcase}
 \newenvironment{subcase}[1][\unskip]{\refstepcounter{subcase}\bf
 \medskip \noindent \hskip2\parindent Subcase \thesubcase\ #1. \it}{\unskip\upshape}
\numberwithin{subcase}{case}
\crefformat{subcase}{Subcase~#2#1#3}

\newcommand{\integer}{\mathbb{Z}}

\newcommand{\voltage}{\Pi}
\newcommand{\quot}{\overline}

\DeclareMathOperator{\Cay}{Cay}
\DeclareMathOperator{\wt}{wt}

\newcommand{\edge}{\mathbin{\hbox{\vrule height 2.5 pt depth -1.75 pt width 10pt}}}

\makeatletter
\newcommand{\noprelistbreak}{\smallskip\@nobreaktrue\nopagebreak} 
\makeatother

\makeatletter
\def\cref@thmoptarg[#1]#2#3#4{%
  \ifhmode\unskip\unskip\par\fi%
  \normalfont%
  \trivlist%
  \let\thmheadnl\relax%
  \let\thm@swap\@gobble%
  \thm@notefont{\fontseries\mddefault\upshape}%
  \thm@headpunct{.}
  \thm@headsep 5\p@ plus\p@ minus\p@\relax%
  \thm@space@setup%
  #2
  \@topsep \thm@preskip               
  \@topsepadd \thm@postskip           
  \def\@tempa{#3}\ifx\@empty\@tempa%
    \def\@tempa{\@oparg{\@begintheorem{#4}{}}[]}%
  \else%
    \refstepcounter[#1]{#3}
    \def\dth@counter{} 
    \def\@tempa{\@oparg{\@begintheorem{#4}{\csname the#3\endcsname}}[]}%
  \fi%
  \@tempa}
\makeatother

\usepackage{color}
\newcommand{\refnote}[1]{\marginpar{%
	\color{blue}
	\vskip0.5\baselineskip
	\vbox to 0pt{\vss
	$\begin{pmatrix} \text{note} \\[-2pt] \text{\ref{#1}} \end{pmatrix}$%
	\vskip -4pt}}}
\newcommand{\refnotes}[2]{\marginpar{%
	\color{blue}
	\vskip\baselineskip\vbox to 0pt{\vss
	$\begin{pmatrix} \text{notes} \\[-2pt] \text{\ref{#1}} \\[-2pt] \text{\ref{#2}} \end{pmatrix}$%
	\vskip -4pt}}}
\theoremstyle{definition}
\newtheorem{aid}[equation]{}
\newcommand{\oldendaid}{}
\let\oldendaid=\endaid
\renewcommand{\endaid}{\oldendaid\bigskip\hrule width\textwidth \filbreak\bigskip}
\crefformat{aid}{Note~#2#1#3}

\newcommand{\bya}[1]{&\stackrel{\textstyle \quot a}{\longrightarrow}&#1}
\newcommand{\byai}[1]{&\stackrel{\textstyle \quot a{}^{-1}}{\longrightarrow}&#1}
\newcommand{\byb}[1]{&\stackrel{\textstyle \quot b}{\longrightarrow}&#1}
\newcommand{\bybi}[1]{&\stackrel{\textstyle \quot b{}^{-1}}{\longrightarrow}&#1}
\newcommand{\byc}[1]{&\stackrel{\textstyle \quot c}{\longrightarrow}&#1}
\newcommand{\byci}[1]{&\stackrel{\textstyle \quot c{}^{-1}}{\longrightarrow}&#1}

\begin{document}

\title{ \textbf{Infinitely Many Nonsolvable Groups \\ Whose Cayley Graphs are Hamiltonian}}

\author{Dave Witte Morris%
\thanks{Dave.Morris@uleth.ca, 
\href{http://people.uleth.ca/~dave.morris/}{{\textsf{http://people.uleth.ca/\!$\sim$dave.morris/}}}}}
\affil{Department of Mathematics and Computer Science, University of Lethbridge, 
\\ Lethbridge, Alberta, T1K\,6R4, Canada}

\maketitle

\begin{abstract}
We show there are infinitely many finite groups~$G$, such that every connected Cayley graph on~$G$ has a hamiltonian cycle, and $G$ is not solvable.
Specifically, we show that if $A_5$~is the alternating group on five letters, and $p$~is any prime, such that $p \equiv 1 \pmod{30}$, then every connected Cayley graph on the direct product $A_5 \times \integer_p$ has a hamiltonian cycle.

\\
\textbf{2010 AMS Classification.} 05C25, 05C45
\end{abstract}

\section{Introduction}

It has been conjectured that every connected Cayley graph on every finite group has a hamiltonian cycle (unless the graph has less than three vertices). In support of this conjecture, the literature provides numerous infinite families of finite groups~$G$, for which it is known that every connected Cayley graph on~$G$ has a hamiltonian cycle. (See \cite{M2Slovenian-LowOrder} and its references for more information.) However, it seems that the union of these families contains only finitely many groups that are not solvable. This note puts an end to that unsatisfactory state of affairs:

\begin{prop}
There are infinitely many finite groups~$G$, such that every connected Cayley graph on~$G$ has a hamiltonian cycle, and $G$ is not solvable.
\end{prop}

Since the alternating group~$A_5$ (of order~$60$) is a nonabelian simple group, and is therefore not solvable, the above is an immediate consequence of the following more specific result.

\begin{prop} \label{A5xZp}
If $p$ is a prime, such that $p \equiv 1 \pmod{30}$, then every connected Cayley graph on the direct product $A_5 \times \integer_p$ has a hamiltonian cycle.
\end{prop}

The proof is based on a case-by-case analysis of Cayley graphs of the group~$A_5$. Most of the hamiltonian cycles were found by computer search (using a fairly naive backtracking algorithm).

\begin{rem}
Rather than merely groups that are not solvable, it would be much more interesting to find infinitely many finite, \emph{simple} groups~$G$, such that every connected Cayley graph on~$G$ has a hamiltonian cycle. Regrettably, the known methods seem to be hopelessly inadequate for this problem.
\end{rem}

\section{Preliminaries}

\begin{defn}
Let $S$ be a subset of a finite group~$G$. The \emph{Cayley graph} of~$G$ with respect to the connection set~$S$ is the graph $\Cay(G;S)$ whose vertices are the elements of~$G$, and with edges $g \edge gs$ and $g \edge g s^{-1}$, for each $g \in G$ and $s \in S$.
\end{defn}

\begin{notation}
Suppose $S$ is a subset of a finite group~$G$.
For $s_1,\ldots,s_m \in S \cup S^{-1}$, we use $(s_i)_{i=1}^m = (s_1,\ldots,s_m)$ to denote the walk in $\Cay(G;S)$ that visits (in order), the vertices
	$$ e, \, s_1, \ s_1 s_2, \, s_1 s_2 s_3, \ \ldots, \, s_1s_2\cdots s_m .$$
We use $(s_1,\ldots,s_m)^k$ to denote the concatenation of~$k$ copies of the sequence $(s_i)_{i=1}^m$, and the following illustrates other notations that are often useful:
	$$ (a^2,b^{-3},s_i)_{i=1}^3 = (a,a,b^{-1},b^{-1},b^{-1},s_1,a,a,b^{-1},b^{-1},b^{-1},s_2,a,a,b^{-1},b^{-1},b^{-1},s_3) .$$
\end{notation}

\begin{notation} \label{QuotDefn}
We use $\quot{\phantom{x}} \colon A_5 \times \integer_p \to A_5$ to denote the natural projection (so $\quot{(x,y)} = x$).
\end{notation}

Our argument in \cref{MainProof} is based on the same outline as the proof in \cite{M2Slovenian-A5} of the following result.

\begin{lem}[\hskip-0.00001em
{\cite{M2Slovenian-A5}}] \label{A5} %
Every connected Cayley graph on~$A_5$ has a hamiltonian cycle.
\end{lem}

The following result is the reason that the statement of \cref{A5xZp} assumes $p \equiv 1 \pmod{30}$. A much weaker hypothesis would suffice in all other parts of the proof.

\begin{cor} \label{nonminA5}
Let $S$ be a minimal generating set of $A_5 \times \integer_p$, and assume $p \equiv 1 \pmod{30}$. If there exists $a \in S$, such that $\quot{S \smallsetminus \{a\}}$ generates~$A_5$, then $\Cay(A_5 \times \integer_p;S)$ has a hamiltonian cycle.
\end{cor}

\begin{proof}
Since $\gcd(|A_5|,p) = 1$, the minimality of~$S$ implies that $\langle S \smallsetminus \{a\}\rangle = A_5$.\refnote{NotMin-<S>=A5}
 (That is, every element of $S \smallsetminus \{a\}$ projects trivially to~$\integer_p$.) From \cref{A5}, we know there is a hamiltonian cycle $(s_i)_{i=1}^{60}$ in $\Cay( A_5; S \smallsetminus \{a\})$. Since, by assumption, $p - 1$ is divisible by $30 = 2 \cdot 3 \cdot 5$ (and every element of~$A_5$ has order $1$, $2$, $3$, or~$5$), we know $\quot a\,^{p-1}$ is trivial. This means $a^{p-1} \in \integer_p$ (so $a^{p-1}$ centralizes~$A_5$), so it is easy to verify that $(s_{2i-1}, a^{p-1}, s_{2i}, a^{-(p-1)})_{i=1}^{30}$ is a hamiltonian cycle in $\Cay( A_5 \times \integer_p ; S)$.\refnote{NotMin}
\end{proof}

\begin{rem}
For definiteness, we point out that we write our permutations on the left, so $g s(i) = g\bigl( s(i) )$ for $g,s \in A_5$ and $i \in \{1,2,3,4,5\}$.
\end{rem}

The remainder of this section 
records a few easy consequences of the following well-known, elementary observation.

\begin{lem}[``Factor Group Lemma'' {\cite[\S2.2]{WitteGallian-survey}}] \label{FGL}
Suppose
\noprelistbreak
 \begin{itemize}
 \item $N$ is a cyclic, normal subgroup of~$G$,
 \item $(s_i)_{i=1}^m$ is a hamiltonian cycle in $\Cay(G/N;S)$,
 and
 \item the voltage $\voltage (s_i)_{i=1}^m$ generates~$N$.
 \end{itemize}
 Then $(s_1,s_2,\ldots,s_m)^{|N|}$ is a hamiltonian cycle in $\Cay(G;S)$.
 \end{lem}

\begin{cor}[\hskip-.00001pt 
{}{\cite[Cor.~2.11]{M2Slovenian-LowOrder}}] \label{DoubleEdge}
  Suppose
 \begin{itemize}
 \item $N$ is a normal subgroup of~$G$, such that $|N|$ is prime,
 \item the image of $S$ in $G/N$ is a minimal generating set of~$G/N$,
 \item there is a hamiltonian cycle in $\Cay(G/N;S)$,
 and
 \item $s \equiv t \pmod{N}$ for some $s,t \in S \cup S^{-1}$ with $s \neq t$.
 \end{itemize}
 Then there is a hamiltonian cycle in $\Cay(G;S)$.
 \end{cor}

\begin{cor} \label{all2}
Let $S$ be a minimal generating set of $A_5 \times \integer_p$, such that $\quot S$ is a minimal generating set of~$A_5$. If every element of~$\quot S$ has order~$2$, then $\Cay(A_5 \times \integer_p;S)$ has a hamiltonian cycle.
\end{cor}

\begin{notation}
Let $C = (s_i)_{i=1}^m$ be a walk in a Cayley graph $\Cay(G;S)$. For $s \in S$, we use $\wt_C(s)$ to denote the difference between the number of occurrences of~$s$ and the number of occurrences of~$s^{-1}$ in~$C$. (This is the \emph{net weight} of the generator~$s$ in~$C$.)
\end{notation}

\begin{lem} \label{det}
Let $S = \{a_1,\ldots,a_k, b_1,\ldots,b_\ell\}$ be a minimal generating set of $A_5 \times \integer_p$, such that $\quot S$ is a minimal generating set of~$A_5$. Assume 
	\begin{itemize}
	\item $\ell \ge 1$, and $|\quot{a_i}| = 2$ for all~$i$,
	\item $C_1,\ldots,C_\ell$ are hamiltonian cycles in $\Cay(A_5; \quot S)$,
	and
	\item $[\wt_{C_i}(b_j)]$ is the $\ell \times \ell$ matrix whose $(i,j)$ entry is $\wt_{C_i}(b_j)$.
	\end{itemize}
If\/ $\det [\wt_{C_i}(b_j)] \not\equiv 0 \pmod{p}$, then $\Cay(A_5 \times \integer_p;S)$ has a hamiltonian cycle.
\end{lem}

\begin{proof}
We may assume that $a_i \in A_5$ for $1 \le i \le k$, for otherwise \cref{DoubleEdge} applies with $s = a_i$ and $t = a_i^{-1}$. Write $b_i = (\quot{b_i}, v_i)$ (with $v_i \in \integer_p$) for $1 \le i \le \ell$. Since $S$ generates $A_5 \times \integer_p$, the vector $[v_1,\ldots,v_\ell]$ must be nonzero in $(\integer_p)^\ell$. Then, since, by assumption, the matrix $[\wt_{C_i}(b_j)]$ is invertible over~$\integer_p$, this implies $[\wt_{C_i}(b_j)]  [v_1,\ldots,v_\ell]^T \neq \vec 0$ in $(\integer_p)^\ell$, so there is some~$i$, such that $\sum_{j=1}^\ell wt_{C_i}(b_j) \, v_j \neq 0$ in~$\integer_p$. This sum is precisely the voltage $\voltage C_i$ of the walk~$C_i$,\refnote{DetWt}
 so \cref{FGL} provides the desired hamiltonian cycle in $\Cay(A_5 \times \integer_p ; S)$.
\end{proof}

\section{Proof of \texorpdfstring{\cref{A5xZp}}{the main result}}\label{MainProof}

\begin{assumps}
Let $S$ be a minimal generating set of $A_5 \times \integer_p$ (and, in accordance with \cref{QuotDefn}, let $\quot S$ be the image of~$S$ in~$A_5$). 
We may assume $\quot S$ is a minimal generating set of~$A_5$, for otherwise \cref{nonminA5} applies.
We may also assume, for every element~$s$ of~$S$ with $|\quot s| = 2$, that the projection of~$s$ to~$\integer_p$ is trivial, for otherwise \cref{DoubleEdge} applies.
\end{assumps}

\setcounter{case}{0}

\begin{case}
Assume $\#S = 2$.
\end{case}
Write $S = \{a,b\}$.

\begin{subcase}
Assume $|\quot a| = 2$ and $|\quot b| = 3$.
\end{subcase}
By applying an automorphism of~$A_5$, we may assume $\quot a = (1,2)(3,4)$ and $\quot b = (2,4,5)$. \refnotes{why(12)(34)+(245)}{(12)(34)+(245)}
We have the following hamiltonian cycle in $\Cay(A_5; \quot S)$:
	$$ \bigl( (\quot a,\quot b\,^{2})^3, (\quot a,\quot b\,^{-2})^3, (\quot a,\quot b\,^{2},\quot a,\quot b\,^{-2})^2 \bigr)^2 .$$
Each left coset of $\langle \quot b \rangle$ appears as consecutive vertices in this cycle, so we see that\refnote{From3Top}
	$$ \bigl( (a,b^{3p-1})^3, (a,b^{-(3p-1)})^3, (a,b^{3p-1},a,b^{-(3p-1)})^2 \bigr)^2 $$
passes through all of the vertices in each left coset of $\langle b \rangle$, and is therefore a hamiltonian cycle in $\Cay(A_5 \times \integer_p; S)$.


\begin{subcase}
Assume $|\quot a| = 2$ and $\quot |b| = 5$.
\end{subcase}
By applying an automorphism of~$G$, we may assume $\quot b = (1,2,3,4,5)$, and that $\quot a$ is either $(1,2)(3,4)$, $(1,3)(2,4)$, or $(1,4)(2,3)$.\refnote{Why2+(12345)}

	\begin{itemize}  \itemsep=\smallskipamount
	\item For $\quot a = (1,2)(3,4)$, we have the hamiltonian cycle
	\refnote{(12)(34)+(12345)}
		\begin{align*}
		 C = \bigl(
		(\quot a,\quot b&,
		 \quot a,\quot b\,^{4})^2, \quot a,\quot b\,^{2},\quot a,\quot b\,^{-1},\quot a,\quot b\,^{4},\quot a,\quot b, (\quot a,\quot b\,^{2})^3,  
		\\& \quot a,\quot b\,^{-2}, \quot a,\quot b\,^{4},\quot a,\quot b\,^{-2},
		(\quot a,\quot b)^2, \quot a,\quot b\,^{-1},\quot a,\quot b\,^{4},\quot a,\quot b\,^{2}
		\bigr) 
		. \end{align*}
Since $\wt_C(\quot b) = 29 \not\equiv 0 \pmod{p}$, \cref{det} applies.

	\item For $\quot a = (1,3)(2,4)$, we have hamiltonian cycle
	\refnote{(13)(24)+(12345)}
	\begin{align*}
	C = \bigl( 
	\quot a,\quot b\,^{4}&,
	\quot a,\quot b\,^{-1},\quot a,\quot b,\quot a,\quot b\,^{-1},\quot a,\quot b\,^{-4},\quot a,\quot b\,^{-2}, (\quot a,\quot b\,^{-4},\quot a,\quot b\,^{2})^2,
	\\&  \quot a,\quot b\,^{-4}, \quot a,\quot b,\quot a,\quot b\,^{-1},\quot a,\quot b\,^{-4},\quot a,\quot b\,^{-2},\quot a,\quot b\,^{-4},\quot a,\quot b\,^{2} \bigr)
	. \end{align*}
Since $\wt_C(\quot b) = -19 \not\equiv 0 \pmod{p}$, \cref{det} applies.

	\item If $\quot a = (1,4)(2,3)$, then $\quot a$ normalizes~$\quot b$, so $\langle \quot a, \quot b \rangle \neq A_5$, which contradicts the fact that $S$ is a generating set.

	\end{itemize}

\begin{subcase}
Assume $|\quot a| = |\quot b| = 3$.
\end{subcase}
By applying an automorphism of~$G$, we may assume $\quot a = (1,2,3)$ and $\quot b = (3,4,5)$. We have the hamiltonian cycle%
	\refnotes{Why123+345}{(123)+(345)}
	\begin{align*}
	C_1 = \bigl( \quot b,\quot a,\quot b\,^{2}&,
	\quot a\,^{2},\quot b\,^{-2},\quot a\,^{-2},\quot b\,^{2},\quot a\,^{2},\quot b\,^{2},\quot a\,^{-2},\quot b\,^{-2},\quot a\,^{-2},\quot b\,^{2},\quot a\,^{2},\quot b\,^{-2},\quot a,\quot b\,^{-2},
	\\& \quot a\,^{-2},\quot b\,^{2},\quot a,\quot b,\quot a\,^{-1},\quot b\,^{-2},\quot a\,^{2},\quot b\,^{2},\quot a\,^{2},\quot b\,^{-2},\quot a\,^{-2},\quot b\,^{-1},\quot a,\quot b,\quot a\,^{2},\quot b\,^{-1},\quot a\,^{-2},\quot b\,^{-1},\quot a  
	\bigr) . \end{align*}
Note that $\wt_{C_1}(\quot a) = 4$ and $\wt_{C_1}(\quot b) = 0$. Conjugation by the permutation $(1,4)(2,5)$ interchanges $\quot a$ and~$\quot b$, and therefore yields a hamiltonian cycle $C_2$ with $\wt_{C_2}(\quot a) = 0$ and $\wt_{C_2}(\quot b) = 4$. Since $\det \left[\begin{smallmatrix} 4 & 0 \\ 0 & 4 \end{smallmatrix}\right] = 16 \not\equiv 0 \pmod{p}$, \cref{det} applies.

\begin{subcase}
Assume $|\quot a| = 3$ and $|\quot b| = 5$.
\end{subcase}
By applying an automorphism of~$G$ (and perhaps replacing $\quot a$ with its inverse), we may assume $\quot b = (1,2,3,4,5)$, and that $\quot a$ is either $(1,2,3)$ or $(1,2,4)$.\refnote{Why123+12345}

	\begin{itemize}
	\item If $\quot a = (1,2,3)$, then we have the hamiltonian cycles
	\refnotes{(123)+(12345)C1}{(123)+(12345)C2}
	\begin{align*}
 C_1 = &\bigl( \quot a\,^{2},\quot b\,^{-2},\quot a\,^{-2},\quot b\,^{-1},\quot a\,^{-1},\quot b\,^{3}, (\quot a\,^{2},\quot b)^2, \quot a,\quot b\,^{-1}, (\quot a\,^{2},\quot b)^2, \quot a\,^{2},\quot b\,^{-2},\quot a\,^{2},\quot b,\quot a,\quot b\,^{-1},
 	\\& \qquad 
	\quot a\,^{-2},\quot b,\quot a\,^{2},\quot b\,^{-1},\quot a\,^{-2},\quot b,\quot a\,^{-1},\quot b\,^{-2},\quot a\,^{-1},\quot b\,^{-1},
	\quot a\,^{-2},\quot b,\quot a\,^{2}, (\quot b\,^{-1},\quot a\,^{-2})^2, \quot b \bigr) \\
 	\intertext{and}
 C_2 = &\bigl( \quot a\,^{2},\quot b\,^{-1}, (\quot a\,^{2},\quot b)^2, \quot a\,^{2},\quot b\,^{-1},\quot a\,^{-2},\quot b,\quot a\,^{-2},\quot b\,^{-1},\quot a\,^{2},\quot b,\quot a\,^{2},\quot b\,^{-1},\quot a\,^{-2},\quot b\,^{2}, 
 	\\& \qquad
	(\quot a\,^{-2},\quot b\,^{-1})^2, (\quot a\,^{-2},\quot b)^2,\quot a\,^{2},\quot b,\quot a\,^{2},\quot b\,^{-1},\quot a,\quot b,\quot a\,^{-2},\quot b,\quot a\,^{2}, (\quot b\,^{-1},\quot a\,^{-2})^2, \quot b \bigr) 
	 \end{align*}
Then $\bigl[ \wt_{C_1}(\quot a), \wt_{C_1}(\quot b) \bigr] = [5,-1]$ 
and $\bigl[ \wt_{C_2}(\quot a), \wt_{C_2}(\quot b) \bigr] = [-1,3]$. 
Since $\det\left[\begin{smallmatrix} 5 & -1 \\ -1 & 3 \end{smallmatrix} \right] = 14 \not\equiv 0 \pmod{p}$, \cref{det} applies.

	\item  If $\quot a = (1,2,4)$, then we have the hamiltonian cycles
	\refnotes{(124)+(12345)C1}{(124)+(12345)C2}
	\begin{align*}
	 C_1 = &\bigl( \quot a\,^{2},\quot b\,^{-2},\quot a\,^{-2},\quot b\,^{-1},\quot a\,^{-1},\quot b,\quot a\,^{2},\quot b\,^{-1},\quot a\,^{-2},\quot b,
 	\quot a\,^{-2},\quot b\,^{-2},\quot a\,^{-2},\quot b\,^{-1},\quot a\,^{-1},\quot b,\quot a\,^{2},\quot b\,^{-1},\quot a\,^{2},\quot b,
 	\\& \qquad 
	\quot a\,^{-2},\quot b\,^{-1},\quot a\,^{2},\quot b,\quot a\,^{2},\quot b\,^{-1},\quot a\,^{-2},\quot b,\quot a\,^{-2},\quot b\,^{-2},\quot a\,^{-2},
	\quot b\,^{-1},\quot a\,^{-1},\quot b,\quot a\,^{2},\quot b\,^{-1}, (\quot a\,^{-2},\quot b)^2 \bigr) \\
 	\intertext{and}
	 C_2 = &	\bigl( \quot a\,^{2},\quot b\,^{-2},\quot a\,^{-2},\quot b\,^{-1},\quot a\,^{-1},\quot b, (\quot a\,^{2},\quot b\,^{-1})^2, \quot a\,^{2},
	 \quot b\,^{-2},\quot a\,^{-2},\quot b\,^{-1},\quot a\,^{-1},\quot b,\quot a\,^{2},\quot b\,^{-1},\quot a\,^{2},\quot b,
 	\\& \qquad
	\quot a\,^{-2},\quot b\,^{-1},\quot a\,^{2},\quot b,\quot a\,^{2},\quot b\,^{-1},\quot a\,^{-2},\quot b,\quot a\,^{-2},\quot b\,^{-2},\quot a\,^{-2},\quot b\,^{-1},
	\quot a\,^{-1},\quot b,\quot a\,^{2},\quot b\,^{-1}, (\quot a\,^{-2},\quot b)^2 \bigr) 
	 \end{align*}
Then $\bigl[ \wt_{C_1}(\quot a), \wt_{C_1}(\quot b) \bigr] = [-9,-5]$ 
and $\bigl[ \wt_{C_2}(\quot a), \wt_{C_2}(\quot b) \bigr] = [-1,-7]$. 
Since $\det\left[\begin{smallmatrix} -9 & -5 \\ -1 & -7 \end{smallmatrix} \right] = 58 \not\equiv 0 \pmod{p}$, \cref{det} applies.

	\end{itemize}

\begin{subcase}
Assume $|\quot a| = |\quot b| = 5$.
\end{subcase}
By applying an automorphism of~$A_5$, we may assume $\quot a = (1,2,3,4,5)$. 
From Sylow's Theorems, we know that $A_5$ has precisely six Sylow $5$-subgroups. One of them is $\langle \quot a \rangle$, and $\langle \quot a \rangle$ acts transitively on the other~$5$. So we may assume, after conjugating by a power of~$\quot a$, that $\langle \quot b \rangle = \langle (1,2,3,5,4) \rangle$. Then, by replacing $b$ with its inverse if necessary, we may assume $\quot b$ is either $(1,2,3,5,4)$ or $(1,3,4,2,5)$.

\begin{itemize}

	\item  If $\quot b = (1,2,3,5,4)$, then we have the hamiltonian cycle
	\refnote{(12345)+(12354)}
	\begin{align*}
	 C_1 = &\bigl( \quot b,\quot a\,^{-1},\quot b,\quot a,\quot b\,^{4},\quot a,\quot b\,^{2},\quot a\,^{-1},\quot b\,^{2},\quot a\,^{-1},\quot b\,^{-1},\quot a\,^{-1},\quot b,
	 \quot a\,^{-1},\quot b,\quot a,\quot b\,^{2},\quot a\,^{-1},\quot b\,^{2},\quot a\,^{-1},\quot b\,^{-1},
 	\\& \qquad 
	\quot a,\quot b\,^{-2},\quot a\,^{-1},\quot b\,^{-4},\quot a\,^{-1},\quot b\,^{-1},\quot a,\quot b\,^{-4},\quot a,\quot b\,^{-2},\quot a\,^{-1},\quot b\,^{-1},
	\quot a,\quot b\,^{-2},\quot a\,^{-1},\quot b\,^{4},\quot a,\quot b\,^{-2},\quot a \bigr)
	 , \end{align*}
with $\bigl[ \wt_{C_1}(\quot a), \wt_{C_1}(\quot b) \bigr] = [-2,0]$. Conjugating by the permutation $(4,5)$ interchanges $\quot a$ and~$\quot b$, and therefore yields a hamiltonian cycle~$C_2$ with
 $\bigl[ \wt_{C_2}(\quot a), \wt_{C_2}(\quot b) \bigr] = [0,-2]$. 
Since $\det\left[\begin{smallmatrix} -2 & 0 \\ 0 & -2 \end{smallmatrix} \right] = 4 \not\equiv 0 \pmod{p}$, \cref{det} applies.

	\item  If $\quot b = (1,3,4,2,5)$, then we have the hamiltonian cycles
	\refnotes{(12345)+(13425)C1}{(12345)+(13425)C2}
	\begin{align*}
	 C_1 = &\bigl( \quot a\,^{4},\quot b\,^{-1},\quot a\,^{-4},\quot b,\quot a\,^{4},\quot b\,^{-1},\quot a\,^{2},\quot b,\quot a\,^{-2},\quot b,\quot a\,^{-1},\quot b,
	 	\quot a\,^{-4},\quot b\,^{-1},
 	\\& \qquad 
	\quot a\,^{-1},\quot b\,^{-1},\quot a\,^{4},\quot b, (\quot a\,^{-4},\quot b\,^{-1})^2, (\quot a\,^{2},\quot b)^2,\quot a\,^{4}, (\quot b\,^{-1},\quot a)^2,\quot b \bigr) \\
 	\intertext{and}
	 C_2 = &	\bigl( \quot a\,^{4},\quot b\,^{-1},\quot a\,^{-4},\quot b,\quot a\,^{4},\quot b\,^{-1},\quot a\,^{2},\quot b,\quot a\,^{-1},\quot b\,^{-1},\quot a\,^{-4},
	 \quot b\,^{-1},\quot a,\quot b\,^{-1},
 	\\& \qquad
	\quot a\,^{-2},\quot b\,^{-1},\quot a\,^{4},\quot b,(\quot a\,^{-4},\quot b\,^{-1})^2, (\quot a\,^{2},\quot b)^2,\quot a\,^{4},(\quot b\,^{-1},\quot a)^2,\quot b \bigr) 
	 \end{align*}
Then $\bigl[ \wt_{C_1}(\quot a), \wt_{C_1}(\quot b) \bigr] = [4,0]$ 
and $\bigl[ \wt_{C_2}(\quot a), \wt_{C_2}(\quot b) \bigr] = [6,-4]$. 
Since $\det\left[\begin{smallmatrix} 4 & 0 \\ 6 & -4 \end{smallmatrix} \right] = -16 \not\equiv 0 \pmod{p}$, \cref{det} applies.

\end{itemize}

\begin{case}
Assume $\#S \ge 3$.
\end{case}
Since $S$ is minimal, it is easy to see that $\#S = 3$, so we may write $S = \{a,b,c\}$ with $|\quot a| \le |\quot b| \le |\quot c|$.\refnote{S<4}

\begin{subcase}
Assume $|\quot c| = 5$.
\end{subcase}
Since $\quot a$ and~$\quot b$ cannot both normalize~$\langle \quot c \rangle$ (but every proper subgroup of~$A_5$ whose order is divisible by~$5$ has order $5$ or~$10$), we see that either $\langle \quot a, \quot c \rangle = A_5$ or $\langle \quot b,\quot c \rangle = A_5$, which contradicts the minimality of~$S$.

\begin{subcase}
Assume $|\quot a| = |\quot b| = |\quot c| = 3$.
\end{subcase}
By applying an automorphism of~$A_5$, and perhaps replacing some generators by their inverses, we may assume\refnote{Why(125)(135)(145)}
	$$\quot S = \{(1,2,5), (1,3,5), (1,4,5) \}
	= \{\, (1,j,n) \mid 1 < j < n \,\} \text{\quad for $n = 5$} .$$
So \cite[App.~D]{GouldRoth-1jnSequencings} provides the following hamiltonian cycle in $\Cay(A_5; \quot{S})$:\refnote{(125)+(135)+(145)}
	$$R_1 = \Bigl( \bigl( (\quot a\,^2, \quot b)^2, \quot a\,^2, \quot c \bigr)^2,
	\quot a, \quot b, (\quot b, \quot a\,^2)^2, \quot c\,^2, \quot a\,^2, \quot c, \quot b, \quot a, \quot b, \quot c, \quot a, (\quot c, \quot b\,^2)^2, 
	(\quot a\,^2, \quot b)^2, \quot a, \quot c\,^2, \quot a, \quot b\,^2, (\quot a, \quot c\,^2)^2 \Bigr) 	
	. $$
We have 
	$[\wt_{R_1}(\quot a), \wt_{R_1}(\quot b), \wt_{R_1}(\quot c)]
	= [29,17,14]$.
Conjugating by $(2,3,4)$ and $(2,3,4)^2$ cyclically permutes $\{\quot a, \quot b, \quot c\}$, and therefore yields hamiltonian cycles $R_2$ and~$R_3$, such that 
	$$ \text{$[\wt_{R_2}(\quot a), \wt_{R_2}(\quot b), \wt_{R_2}(\quot c)]
	= [17,14,29]$
and
	$[\wt_{R_3}(\quot a), \wt_{R_3}(\quot b), \wt_{R_3}(\quot c)]
	= [14,29,17]$} .$$
Since 
	$$ \det \begin{bmatrix} 29 & 17 & 14 \\ 17 & 14 & 29 \\ 14 & 29 & 17 \end{bmatrix} 
	= 11,340 = 2^2 \cdot 3^4 \cdot 5 \cdot 7 \not\equiv 0 \pmod{p} ,$$
\cref{det} applies.

\begin{subcase}
Assume $|\quot a| = 2$ and $|\quot b| = |\quot c| = 3$.
\end{subcase}
Since $\langle \quot b, \quot c \rangle$ is a proper subgroup of~$A_5$ that is generated by two elements of order~$3$, it is conjugate to~$A_4$. So we may assume $\quot b = (1,2,3)$ and $\quot c = (1,2,4)$.\refnote{Why(12)(45)+(123)+(124)}
 And then, since $S$ is minimal, we may assume $\quot a = (1,2)(4,5)$ (perhaps after conjugating by an element of~$S_5$ that interchanges $\quot b$ and~$\quot c$). 

We have the hamiltonian cycles
	\refnotes{(12)(45)+(123)+(124)C1}{(12)(45)+(123)+(124)C2}
	\begin{align*}
	 C_1 = &\bigl( \quot a,\quot c\,^{-1},\quot a,\quot b,\quot a,\quot c,\quot a,\quot b\,^{2},\quot a,\quot b,\quot c,\quot b\,^{-1},\quot a,\quot b\,^{-2},\quot a,\quot c\,^{-1},\quot a,\quot c\,^{-1},\quot b\,^{2},\quot c,\quot b\,^{-2},\quot a,\quot b\,^{-2},\quot c,\quot a,\quot b,\quot c\,^{-1},
 	\\& \qquad 
	\quot b\,^{-1},\quot a,\quot c\,^{-1},\quot a,\quot b\,^{-2},\quot a,\quot b\,^{-1},\quot c,\quot a,\quot b\,^{2},\quot a,\quot b,\quot c,\quot b\,^{-1},\quot a,\quot c\,^{-1},\quot b\,^{2},\quot c,\quot a,\quot b,\quot c\,^{-1},\quot a,\quot b\,^{-1},\quot a,\quot b \bigr) \\
 	\intertext{and}
	 C_2 = &	\bigl( \quot a,\quot c\,^{-1},\quot a,\quot b,\quot a,\quot c,\quot a,\quot b\,^{2},\quot a,\quot b,\quot c,\quot b\,^{-1},\quot a,\quot c\,^{-1},\quot b\,^{2},\quot c,\quot a,\quot b,\quot c\,^{-1},\quot b\,^{-1},\quot a,\quot b\,^{-2},\quot a,\quot c\,^{-1},\quot b,\quot a,\quot b\,^{2},
 	\\& \qquad
	\quot a,\quot c,\quot a,\quot b,\quot c\,^{-1},\quot b\,^{-1},\quot a,\quot c\,^{-1},\quot a,\quot c\,^{-1},\quot b\,^{2},\quot c,\quot b\,^{-2},\quot a,\quot b\,^{-2},\quot c,\quot a,\quot b\,^{2},\quot a,\quot b,\quot c\,^{-1},\quot a,\quot b\,^{-1},\quot a,\quot b \bigr) 
	 \end{align*}
Then $\bigl[ \wt_{C_1}(\quot b), \wt_{C_1}(\quot c) \bigr] = [1,0]$ 
and $\bigl[ \wt_{C_2}(\quot b), \wt_{C_2}(\quot c) \bigr] = [7,-2]$. 
Since $\det\left[\begin{smallmatrix} 1 & 0 \\ 7 & -2 \end{smallmatrix} \right] = -2 \not\equiv 0 \pmod{p}$, \cref{det} applies.

\begin{subcase}
Assume $|\quot a| = |b| = 2$ and $|\quot c| = 3$.
\end{subcase}
We may assume $\quot c = (1,2,3)$. 

\begin{itemize}  \itemsep=\smallskipamount

\item If $\quot a$ interchanges $4$ and~$5$, then we may assume $\quot a = (1,2)(4,5)$. And then we may assume $\quot b$ is either 
$(1,2)(3,4)$ or $(1,3)(2,4)$.\refnote{Why(12)(45)+(12)(34)+(124)}

\begin{itemize}

\item If $\quot b = (1,2)(3,4)$, we have the hamiltonian cycle
\refnote{(12)(45)+(12)(34)+(123)}
\begin{align*}
 C =  & \Bigl( \quot a,\quot c\,^{-1}, (\quot a,\quot b)^2, \quot a,\quot c\,^{-1},\quot a,\quot b,\quot c\,^{-1},\quot b,\quot c, (\quot b,\quot a)^2, \quot c, (\quot a,\quot b)^2,\quot a,\quot c, 
 (\quot a,\quot b)^2, 
\\& \quad \quot a,\quot c\,^{2}, \quot b,\quot c\,^{-1},\quot b, (\quot a,\quot b)^2, \quot c\,^{-2},
\bigl( (\quot a,\quot b)^2, \quot a,\quot c\,^{-1} \bigr)^2, \quot c\,^{-1},\quot b,\quot a,
\quot c\,^{-1},(\quot a,\quot b)^2 \Bigr)
, \end{align*}
with $\wt_C(\quot c) = -5 \not\equiv 0 \pmod{p}$, so \cref{det} applies.

\item If $\quot b = (1,3)(2,4)$, we have the hamiltonian cycle
\refnote{(12)(45)+(13)(24)+(123)}
\begin{align*}
C =  & \Bigl( (\quot a,\quot b)^4,\quot c,(\quot a,\quot b)^4,\quot a,\quot c\,^{-1},\quot b,\quot a,\quot b,\quot c\,^{-1}, (\quot b,\quot a)^3,\quot c\,^{-1},
\\& \quad \quot b, (\quot a,\quot b)^2, \quot c,\quot b,\quot a,\quot c\,^{-1}, (\quot b,\quot a)^4,\quot b,\quot c, (\quot a,\quot b)^4, \quot a,\quot c\,^{-1},\quot b \Bigr)
, \end{align*}
with $\wt_C(\quot c) = -2 \not\equiv 0 \pmod{p}$, so \cref{det} applies.

\end{itemize}

\item If neither $\quot a$ nor~$\quot b$ interchanges $4$ and~$5$, then one of them must fix~$4$, and the other must fix~$5$. We may assume $\quot a = (1,2)(3,4)$. And then we may assume $\quot b$ is either $(1,2)(3,5)$ or $(1,3)(2,5)$.\refnote{Why(12)(34)+(12)(35)+(123)}

\begin{itemize}
\item If $\quot b = (1,2)(3,5)$, we have the hamiltonian cycle
\refnote{(12)(34)+(12)(35)+(123)}
\begin{align*}
 C =  &\Bigl( \quot a,\quot c\,^{-1}, \quot a,\quot c,\quot b, (\quot a,\quot b)^2, \quot c,
(\quot a, \quot b)^2, (\quot a,\quot c)^2, \quot a,\quot b,\quot a,\quot c\,^{-2},\quot a,
\quot c\,^{-1},\quot b,\quot a,
\\& \quad \quot c,\quot b, \quot a,\quot b,\quot c\,^{2}, \bigl( (\quot a,\quot b)^2, \quot a,\quot c \bigr)^2,
(\quot a,\quot c\,^{-1})^2, (\quot a, \quot b)^2, \quot a,\quot c\,^{-2}, (\quot a,\quot b)^2 \Bigr)
, \end{align*}
with $\wt_C(\quot c) = 1 \not\equiv 0 \pmod{p}$, so \cref{det} applies.

\item If $\quot b = (1,3)(2,5)$, we have the hamiltonian cycle
\refnote{(12)(34)+(13)(25)+(123)}
\begin{align*}
 C =  & \Bigl( \quot a,\quot b,\quot c\,^{-1}, (\quot a,\quot b)^2,
\quot a,\quot c,\quot b, (\quot a,\quot b)^4,
\quot c, (\quot a,\quot b,\quot a,\quot c\,^{-1})^2,
\quot b,\quot a,\quot b,\quot c\,^{2},
\\& \quad (\quot a,\quot b)^2, \quot a,\quot c, 
\bigl( (\quot b,\quot a)^2, \quot c\,^{-1} \bigr)^2,
(\quot b,\quot a)^2,\quot b,\quot c\,^{-1},
(\quot b,\quot a)^2,\quot c\,^{-1},\quot b \Bigr)
, \end{align*}
with $\wt_C(\quot c) = -2 \not\equiv 0 \pmod{p}$, so \cref{det} applies.

\end{itemize}
\end{itemize}

\begin{subcase}
Assume $|\quot a| = |\quot b| = |\quot c| = 2$.
\end{subcase}
Since all generators are of order~$2$, \cref{all2} applies.
\qed


 
 \newpage
 
 \thispagestyle{empty}

\null \cleardoublepage 

\setcounter{page}{1}
\renewcommand{\thepage}{A-\arabic{page}}

\lhead{Appendix: Notes to aid the referee} 

\thispagestyle{plain}

\begin{appendix}

\section*{Appendix: Notes to aid the referee
	\\ \normalfont \emph{Infinitely many nonsolvable groups whose Cayley graphs are hamiltonian}
	\\ {\small by Dave Witte Morris}}
	\setcounter{equation}{0}
	\renewcommand{\theequation}{A.\arabic{equation}}

\begin{aid} \label{NotMin-<S>=A5}
Since $\gcd(|A_5|,p) = 1$, we have $\quot g \in \langle g \rangle$ for every $g \in A_5 \times \integer_p$. Therefore $A_5 = \langle \quot{S \smallsetminus \{a\}} \rangle \subseteq \langle S \smallsetminus \{a\} \rangle$. Since the minimality of~$S$ implies $a \notin \subseteq \langle S \smallsetminus \{a\} \rangle$, we conclude that $\langle S \smallsetminus \{a\} \rangle = A_5$.
\end{aid}

\begin{aid} \label{NotMin}
Since $\langle S \rangle = A_5 \times \integer_p$, and $S \subseteq A_5$, we know that $a$ projects nontrivially to~$\integer_p$. Since $\gcd(|A_5|,p) = 1$, this implies $\integer_p \subseteq \langle a \rangle$, so every element~$g$ of $A_5 \times \integer_p$ can be written (uniquely) in the form $g = x a^r$ with $x \in A_5$ and $0 \le r \le p-1$. Since $(s_i)_{i=1}^{60}$ is a hamiltonian cycle in a Cayley graph on~$A_5$, we have $x = s_1 s_2 \cdots s_k$, for some~$k$ with $0\le k < 60$. If $k = 2i - 1$ is odd, then
	\begin{align*}
	 g 
	 &= x a^r 
	\\&= \left( \prod_{j=1}^{i-1} s_{2j-1} s_{2j} \right) s_{2i-1} a^r
	\\&= \left( \prod_{j=1}^{i-1} (s_{2j-1} a^{p-1} s_{2j} a^{-(p-1)} \right) s_{2i-1} a^r
	&& \text{(because $a^{p-1} \in \integer_p$ is in the center of $A_5 \times \integer_p$)}
	. \end{align*}
If $k = 2i$ is even, then
	\begin{align*}
	 g 
	 &= x a^r 
	\\&= \left( \prod_{j=1}^{i-1} s_{2j-1} s_{2j} \right) s_{2i-1} s_{2i} a^r
	\\&= \left( \prod_{j=1}^{i-1} s_{2j-1} a^{p-1} s_{2j} a^{-(p-1)} \right) s_{2i-1} a^{p-1} s_{2i} a^{-(p-1-r)}
	. \end{align*}
In either case, we see that $g$ is one of the vertices on the walk $(s_{2i-1}, a^{p-1}, s_{2i}, a^{-(p-1)})_{i=1}^{30}$. Therefore, the walk passes through all of the vertices in $\Cay(A_5 \times \integer_p ; S)$.

Also, note that the walk has the correct length ($60p$) to be a hamiltonian cycle. Finally, by using once again the fact that $a^{p-1}$ is in the center, we see that the terminal vertex of the walk is
	$$ \prod_{i=1}^{30} s_{2i-1} a^{p-1} s_{2i} a^{-(p-1)}
	= \prod_{i=1}^{30} s_{2i-1} s_{2i} 
	= \prod_{i=1}^{60} s_i
	= e ,$$
because $(s_i)_{i=1}^{60}$ is a (hamiltonian) cycle, and therefore has terminal vertex~$e$.
We conclude that the walk is a hamiltonian cycle.
\end{aid}

\begin{aid} \label{DetWt}
Write $C_i = (s_k)_{k=1}^n$, and let $\pi = \prod_{k=1}^n s_k$ be the voltage of~$C_i$. Since $C_i$ is a (hamiltonian) cycle in $\Cay(A_5;\quot S)$, we know that $\quot \pi$ is trivial, so $\pi = \sum_{k = 1}^n s_k^*$, where $s_k^*$ is the projection of~$s_k$ to~$\integer_p$. Noting that $(s^{-1})^* = - s^*$ for $s \in S$, we have
	\begin{align*}
	 \pi 
	 &= \sum_{k = 1}^n s_k^*
	\\&= \sum_{s \in S \cup S^{-1}} (\text{\# occurrences of $s$ in~$C_i$}) \cdot s^*
	\\&= \sum_{s \in S} \bigl( (\text{\# occurrences of $s$ in~$C_i$})
		- (\text{\# occurrences of $s^{-1}$ in~$C_i$}) \bigr) \cdot s^*
	\\&= \sum_{s \in S} \wt_{C_i}(s) \, s^*
	\\&= \sum_{j=1}^\ell \wt_{C_i}(b_j) \, v_j
	. \end{align*}
\end{aid}

\begin{aid} \label{why(12)(34)+(245)}
We may assume $\quot a = (1,2)(3,4)$, since every element of order~$2$ in~$A_5$ is conjugate to this. Then, in order for $\langle \quot a, \quot b \rangle$ to be transitive, the support of~$\quot b$ must contain an element of each cycle of~$\quot a$ (including the $1$-cycle $(5)$). So we may assume $\quot b = (2,4,5)$ (after conjugating by $(1,2)$ and/or $(3,4)$, if necessary).
\end{aid}

\begin{aid} \label{(12)(34)+(245)}
A hamiltonian cycle in $\Cay(A_5; \quot a, \quot b)$ with $\quot a = (1,2)(3,4)$ and $\quot b = (2,4,5)$.
\begin{align*}
\begin{array}{ccccccccccc}
&&e\bya{(1,2)(3,4)}\byb{(1,2,3,4,5)}\byb{(1,2,5,3,4)} \\
\bya{(1,5,3)}\byb{(1,5,2,4,3)}\byb{(1,5,4,2,3)}\bya{(1,3,2,5,4)} \\
\byb{(1,3,2)}\byb{(1,3,2,4,5)}\bya{(1,4,2,3,5)}\bybi{(1,4,3,5,2)} \\
\bybi{(1,4)(3,5)}\bya{(1,2,4,5,3)}\bybi{(1,2,3)}\bybi{(1,2,5,4,3)} \\
\bya{(1,5,4)}\bybi{(1,5)(2,4)}\bybi{(1,5,2)}\bya{(2,5)(3,4)} \\
\byb{(2,3,4)}\byb{(3,4,5)}\bya{(1,2)(3,5)}\bybi{(1,2,3,5,4)} \\
\bybi{(1,2,4,3,5)}\bya{(1,4,5)}\byb{(1,4)(2,5)}\byb{(1,4,2)} \\
\bya{(2,4,3)}\bybi{(2,5,3)}\bybi{(2,3)(4,5)}\bya{(1,3,5,4,2)} \\
\byb{(1,3,5)}\byb{(1,3,5,2,4)}\bya{(1,4,5,2,3)}\byb{(1,4,2,5,3)} \\
\byb{(1,4,3)}\bya{(1,2,4)}\byb{(1,2)(4,5)}\byb{(1,2,5)} \\
\bya{(1,5)(3,4)}\bybi{(1,5,3,4,2)}\bybi{(1,5,2,3,4)}\bya{(1,3)(2,5)} \\
\bybi{(1,3)(4,5)}\bybi{(1,3)(2,4)}\bya{(1,4)(2,3)}\bybi{(1,4,3,2,5)} \\
\bybi{(1,4,5,3,2)}\bya{(2,4)(3,5)}\byb{(3,5,4)}\byb{(2,3,5)} \\
\bya{(1,3,4,5,2)}\bybi{(1,3,4)}\bybi{(1,3,4,2,5)}\bya{(1,5)(2,3)} \\
\byb{(1,5,3,2,4)}\byb{(1,5,4,3,2)}\bya{(2,5,4)}\bybi{(2,4,5)} \\
\bybi{e}
\end{array}
\end{align*}
\end{aid}

\begin{aid} \label{From3Top}
For convenience, let
	$$ C_1 =  \bigl( (\quot a,\quot b\,^{2})^3, (\quot a,\quot b\,^{-2})^3, (\quot a,\quot b\,^{2},\quot a,\quot b\,^{-2})^2 \bigr)^2 $$
and
	$$ C_2 = \bigl( (a,b^{3p-1})^3, (a,b^{-(3p-1)})^3, (a,b^{3p-1},a,b^{-(3p-1)})^2 \bigr)^2 .$$
Note that $\quot{b^{3p-1}} = \quot b{}^2$ (since $|\quot b| = 3$). 

Let $x$ be the terminal vertex of the walk~$C_2$. From the preceding paragraph, we see that $\quot x$ is the terminal vertex of the hamiltonian cycle~$C_1$, so $\quot x$ is trivial. The projection of~$x$ to~$\integer_p$ is also trivial, because $\wt_{C_2}(b) = 0$. Therefore, the walk $C_2$ is closed. Since $C_2$ has the correct length to be a hamiltonian, we need only show that it passes through every element of~$A_5 \times \integer_p$.

From the fact that $\quot{b^{3p-1}} = \quot b{}^2$, we see that the vertices of~$\quot{C_2}$ are precisely the same elements of~$A_5$ as the vertices of~$C_1$; that is, the walk $\quot{C_2}$ passes through every element of~$A_5$. Thus, given any $v \in A_5 \times \integer_p$, the walk~$C_2$ visits some vertex~$w$ with $\quot w = \quot v$; that is, $v$ and~$w$ are in the same coset of~$\integer_p$. Since $\integer_p \subseteq \langle b \rangle$, this implies that $v$ and~$w$ are in the same left coset of~$\langle b \rangle$. Also, since there are never two consecutive appearances of~$a$ in~$C_2$, and every occurrence of~$b$ is contained in a string $b^{3p-1}$, we know that $C_2$ traverses every element of any left coset of~$\langle b \rangle$ that it enters. In particular, $C_2$ traverses every element of the left coset of~$w$, so it passes through~$v$. 
\end{aid}

\begin{aid} \label{Why2+(12345)}
We may assume $\quot b = (1,2,3,4,5)$, after conjugating by some permutation in~$S_5$. Since $|\quot a| = 2$, it has a fixed point, which we may assume is~$5$ (after conjugating by a power of~$\quot b$). So $|\quot a|$ must be either $(1,2)(3,4)$, $(1,3)(2,4)$, or $(1,4)(2,3)$.
\end{aid}

\begin{aid} \label{(12)(34)+(12345)}
A hamiltonian cycle in $\Cay(A_5; \quot a, \quot b)$ with $\quot a = (1,2)(3,4)$ and $\quot b = (1,2,3,4,5)$.
\begin{align*}
\begin{array}{ccccccccccc}
&&e\bya{(1,2)(3,4)}\byb{(2,4,5)}\bya{(1,4,3,5,2)} \\
\byb{(2,5,4)}\byb{(1,5)(2,3)}\byb{(1,3,4)}\byb{(1,2,4,5,3)} \\
\bya{(1,4)(3,5)}\byb{(1,2,5,4,3)}\bya{(1,5,4)}\byb{(1,2,3)} \\
\byb{(1,3,4,5,2)}\byb{(2,4)(3,5)}\byb{(1,4,3,2,5)}\bya{(1,5)(2,4)} \\
\byb{(1,4)(2,3)}\byb{(1,3)(4,5)}\bya{(1,2,3,5,4)}\bybi{(1,4,5)} \\
\bya{(1,2,4,3,5)}\byb{(1,4)(2,5)}\byb{(1,5,4,2,3)}\byb{(1,3,2)} \\
\byb{(3,4,5)}\bya{(1,2)(3,5)}\byb{(2,5)(3,4)}\bya{(1,5,2)} \\
\byb{(2,3,4)}\byb{(1,3,2,4,5)}\bya{(1,4,2,3,5)}\byb{(1,3,2,5,4)} \\
\byb{(1,5,3)}\bya{(1,2,5,3,4)}\byb{(1,5,2,4,3)}\byb{(1,4,2)} \\
\bya{(2,4,3)}\bybi{(1,5,3,4,2)}\bybi{(1,3)(2,5)}\bya{(1,5,2,3,4)} \\
\byb{(1,3)(2,4)}\byb{(1,4,5,3,2)}\byb{(3,5,4)}\byb{(1,2,5)} \\
\bya{(1,5)(3,4)}\bybi{(2,5,3)}\bybi{(1,3,5,4,2)}\bya{(2,3)(4,5)} \\
\byb{(1,3,5)}\bya{(1,2,3,4,5)}\byb{(1,3,5,2,4)}\bya{(1,4,5,2,3)} \\
\bybi{(1,2,4)}\bya{(1,4,3)}\byb{(1,2)(4,5)}\byb{(2,3,5)} \\
\byb{(1,3,4,2,5)}\byb{(1,5,3,2,4)}\bya{(1,4,2,5,3)}\byb{(1,5,4,3,2)} \\
\byb{e}
\end{array}
\end{align*}
\end{aid}

\begin{aid} \label{(13)(24)+(12345)}
A hamiltonian cycle in $\Cay(A_5; \quot a, \quot b)$ with $\quot a = (1,3)(2,4)$ and $\quot b = (1,2,3,4,5)$.
\begin{align*}
\begin{array}{ccccccccccc}
&&e\bya{(1,3)(2,4)}\byb{(1,4,5,3,2)}\byb{(3,5,4)} \\
\byb{(1,2,5)}\byb{(1,5,2,3,4)}\bya{(1,4,3,5,2)}\bybi{(1,2,4,5,3)} \\
\bya{(2,5,3)}\byb{(1,5)(3,4)}\bya{(1,4,2,3,5)}\bybi{(2,4,5)} \\
\bya{(1,3)(2,5)}\bybi{(1,2,3,5,4)}\bybi{(1,4,5)}\bybi{(2,4,3)} \\
\bybi{(1,5,3,4,2)}\bya{(1,4)(3,5)}\bybi{(1,3,2,4,5)}\bybi{(2,3,4)} \\
\bya{(1,4,3)}\bybi{(1,5,3,2,4)}\bybi{(1,3,4,2,5)}\bybi{(2,3,5)} \\
\bybi{(1,2)(4,5)}\bya{(1,3,2,5,4)}\byb{(1,5,3)}\byb{(1,2)(3,4)} \\
\bya{(1,4)(2,3)}\bybi{(1,5)(2,4)}\bybi{(2,5)(3,4)}\bybi{(1,2)(3,5)} \\
\bybi{(1,3)(4,5)}\bya{(2,5,4)}\byb{(1,5)(2,3)}\byb{(1,3,4)} \\
\bya{(1,4,2)}\bybi{(1,5,2,4,3)}\bybi{(1,2,5,3,4)}\bybi{(1,3,5)} \\
\bybi{(2,3)(4,5)}\bya{(1,2,5,4,3)}\byb{(1,5,2)}\bya{(1,3,5,2,4)} \\
\bybi{(1,2,3,4,5)}\bya{(1,4,3,2,5)}\bybi{(2,4)(3,5)}\bybi{(1,3,4,5,2)} \\
\bybi{(1,2,3)}\bybi{(1,5,4)}\bya{(1,3,5,4,2)}\bybi{(1,4,5,2,3)} \\
\bybi{(1,2,4)}\bya{(1,3,2)}\bybi{(1,5,4,2,3)}\bybi{(1,4)(2,5)} \\
\bybi{(1,2,4,3,5)}\bybi{(3,4,5)}\bya{(1,4,2,5,3)}\byb{(1,5,4,3,2)} \\
\byb{e}
\end{array}
\end{align*}
\end{aid}

\begin{aid} \label{Why123+345}
The union of the supports of~$\quot a$ and~$\quot b$ must be all of $\{1,2,3,4,5\}$, since $\langle \quot a, \quot b \rangle$ is transitive. Since each support consists of only three elements, the intersection must be a single element, which we may assume is~$3$. Then, by renumbering, we may assume the support of~$\quot a$ is $\{1,2,3\}$ and the support of~$\quot b$ is $\{3,4,5\}$. Therefore, either $\quot a$ or $\quot a\,^{-1}$ is $(1,2,3)$, and either $\quot b$ or $\quot b\,^{-1}$ is $(3,4,5)$.
\end{aid}

\begin{aid} \label{(123)+(345)}
A hamiltonian cycle~$C_1$ in $\Cay(A_5; \quot a, \quot b)$ with $\quot a = (1,2,3)$ and $\quot b = (3,4,5)$.
\begin{align*}
\begin{array}{ccccccccccc}
&&e\byb{(3,4,5)}\bya{(1,2,4,5,3)}\byb{(1,2,4,3,5)} \\
\byb{(1,2,4)}\bya{(1,4)(2,3)}\bya{(1,3,4)}\bybi{(1,3,5)} \\
\bybi{(1,3)(4,5)}\byai{(2,3)(4,5)}\byai{(1,2)(4,5)}\byb{(1,2)(3,5)} \\
\byb{(1,2)(3,4)}\bya{(2,4,3)}\bya{(1,4,3)}\byb{(1,4,5)} \\
\byb{(1,4)(3,5)}\byai{(1,5,3,2,4)}\byai{(1,2,5,3,4)}\bybi{(1,2,5)} \\
\bybi{(1,2,5,4,3)}\byai{(3,5,4)}\byai{(1,5,4,3,2)}\byb{(1,5,2)} \\
\byb{(1,5,3,4,2)}\bya{(2,4)(3,5)}\bya{(1,4,2,5,3)}\bybi{(1,4)(2,5)} \\
\bybi{(1,4,3,2,5)}\bya{(1,5)(3,4)}\bybi{(1,5,3)}\bybi{(1,5,4)} \\
\byai{(1,3,2,5,4)}\byai{(1,2,3,5,4)}\byb{(1,2,3)}\byb{(1,2,3,4,5)} \\
\bya{(1,3,2,4,5)}\byb{(1,3,5,2,4)}\byai{(1,5,2,3,4)}\bybi{(1,5)(2,3)} \\
\bybi{(1,5,4,2,3)}\bya{(1,3,5,4,2)}\bya{(2,5,4)}\byb{(2,5,3)} \\
\byb{(2,5)(3,4)}\bya{(1,5,2,4,3)}\bya{(1,4,3,5,2)}\bybi{(1,4,5,3,2)} \\
\bybi{(1,4,2)}\byai{(1,3)(2,4)}\byai{(2,3,4)}\bybi{(2,3,5)} \\
\bya{(1,3)(2,5)}\byb{(1,3,4,2,5)}\bya{(1,5)(2,4)}\bya{(1,4,2,3,5)} \\
\bybi{(1,4,5,2,3)}\byai{(2,4,5)}\byai{(1,3,4,5,2)}\bybi{(1,3,2)} \\
\bya{e}
\end{array}
\end{align*}
\end{aid}

\begin{aid} \label{Why123+12345}
We may assume $\quot b = (1,2,3,4,5)$, by replacing it with a conjugate. Then the two fixed points of the $3$-cycle $\quot a$ are either consecutive or are separated by only one element (in circular order). Thus, after conjugating by a power of~$\quot b$, we may assume that the fixed points of~$\quot a$ are either $4$ and~$5$ or $3$ and~$5$. Hence, $\quot a$ is either $(1,2,3)^{\pm1}$ or $(1,2,4)^{\pm1}$.
\end{aid}

\begin{aid} \label{(123)+(12345)C1}
A hamiltonian cycle~$C_1$ in $\Cay(A_5; \quot a, \quot b)$ with $\quot a = (1,2,3)$ and $\quot b = (1,2,3,4,5)$.
\begin{align*}
\begin{array}{ccccccccccc}
&&e\bya{(1,2,3)}\bya{(1,3,2)}\bybi{(1,5,4,2,3)} \\
\bybi{(1,4)(2,5)}\byai{(1,3,5,2,4)}\byai{(1,5,2,3,4)}\bybi{(1,2,5)} \\
\byai{(1,3,5)}\byb{(1,2,5,3,4)}\byb{(1,5,2,4,3)}\byb{(1,4,2)} \\
\bya{(2,3,4)}\bya{(1,3)(2,4)}\byb{(1,4,5,3,2)}\bya{(3,4,5)} \\
\bya{(1,2,4,5,3)}\byb{(1,4,3,5,2)}\bya{(2,5)(3,4)}\bybi{(1,2)(3,5)} \\
\bya{(2,5,3)}\bya{(1,5,3)}\byb{(1,2)(3,4)}\bya{(2,4,3)} \\
\bya{(1,4,3)}\byb{(1,2)(4,5)}\bya{(2,3)(4,5)}\bya{(1,3)(4,5)} \\
\bybi{(1,4)(2,3)}\bybi{(1,5)(2,4)}\bya{(1,4,2,3,5)}\bya{(1,3,4,2,5)} \\
\byb{(1,5,3,2,4)}\bya{(1,4)(3,5)}\bybi{(1,3,2,4,5)}\byai{(1,2,3,4,5)} \\
\byai{(1,4,5)}\byb{(1,2,3,5,4)}\bya{(1,3,2,5,4)}\bya{(1,5,4)} \\
\bybi{(1,4,3,2,5)}\byai{(1,2,4,3,5)}\byai{(1,5)(3,4)}\byb{(1,2,4)} \\
\byai{(1,3,4)}\bybi{(1,5)(2,3)}\bybi{(2,5,4)}\byai{(1,3,5,4,2)} \\
\bybi{(1,4,5,2,3)}\byai{(2,4,5)}\byai{(1,3,4,5,2)}\byb{(2,4)(3,5)} \\
\bya{(1,4,2,5,3)}\bya{(1,5,3,4,2)}\bybi{(1,3)(2,5)}\byai{(2,3,5)} \\
\byai{(1,5,2)}\bybi{(1,2,5,4,3)}\byai{(3,5,4)}\byai{(1,5,4,3,2)} \\
\byb{e}
\end{array}
\end{align*}
\end{aid}

\begin{aid} \label{(123)+(12345)C2}
A second hamiltonian cycle~$C_2$ in $\Cay(A_5; \quot a, \quot b)$ with $\quot a = (1,2,3)$ and $\quot b = (1,2,3,4,5)$.
\begin{align*}
\begin{array}{ccccccccccc}
&&e\bya{(1,2,3)}\bya{(1,3,2)}\bybi{(1,5,4,2,3)} \\
\bya{(1,3,5,4,2)}\bya{(2,5,4)}\byb{(1,5)(2,3)}\bya{(1,3,5)} \\
\bya{(1,2,5)}\byb{(1,5,2,3,4)}\bya{(1,3,5,2,4)}\bya{(1,4)(2,5)} \\
\bybi{(1,2,4,3,5)}\byai{(1,5)(3,4)}\byai{(1,4,3,2,5)}\byb{(1,5,4)} \\
\byai{(1,3,2,5,4)}\byai{(1,2,3,5,4)}\bybi{(1,4,5)}\bya{(1,2,3,4,5)} \\
\bya{(1,3,2,4,5)}\byb{(1,4)(3,5)}\bya{(1,2,5,3,4)}\bya{(1,5,3,2,4)} \\
\bybi{(1,3,4,2,5)}\byai{(1,4,2,3,5)}\byai{(1,5)(2,4)}\byb{(1,4)(2,3)} \\
\byb{(1,3)(4,5)}\byai{(2,3)(4,5)}\byai{(1,2)(4,5)}\bybi{(1,4,3)} \\
\byai{(2,4,3)}\byai{(1,2)(3,4)}\bybi{(1,5,3)}\byai{(2,5,3)} \\
\byai{(1,2)(3,5)}\byb{(2,5)(3,4)}\byai{(1,4,3,5,2)}\byai{(1,5,2,4,3)} \\
\byb{(1,4,2)}\bya{(2,3,4)}\bya{(1,3)(2,4)}\byb{(1,4,5,3,2)} \\
\bya{(3,4,5)}\bya{(1,2,4,5,3)}\bybi{(1,3,4)}\bya{(1,2,4)} \\
\byb{(1,4,5,2,3)}\byai{(2,4,5)}\byai{(1,3,4,5,2)}\byb{(2,4)(3,5)} \\
\bya{(1,4,2,5,3)}\bya{(1,5,3,4,2)}\bybi{(1,3)(2,5)}\byai{(2,3,5)} \\
\byai{(1,5,2)}\bybi{(1,2,5,4,3)}\byai{(3,5,4)}\byai{(1,5,4,3,2)} \\
\byb{e}
\end{array}
\end{align*}
\end{aid}

\begin{aid} \label{(124)+(12345)C1}
A hamiltonian cycle~$C_1$ in $\Cay(A_5; \quot a, \quot b)$ with $\quot a = (1,2,4)$ and $\quot b = (1,2,3,4,5)$.
\begin{align*}
\begin{array}{ccccccccccc}
&&e\bya{(1,2,4)}\bya{(1,4,2)}\bybi{(1,5,2,4,3)} \\
\bybi{(1,2,5,3,4)}\byai{(3,4,5)}\byai{(1,5,3,4,2)}\bybi{(1,3)(2,5)} \\
\byai{(1,4,5,2,3)}\byb{(1,3,5,4,2)}\bya{(3,5,4)}\bya{(1,2,3,5,4)} \\
\bybi{(1,4,5)}\byai{(1,5)(2,4)}\byai{(1,2,5)}\byb{(1,5,2,3,4)} \\
\byai{(2,5)(3,4)}\byai{(1,3,4,5,2)}\bybi{(1,2,3)}\bybi{(1,5,4)} \\
\byai{(2,5,4)}\byai{(1,2)(4,5)}\bybi{(1,4,3)}\byai{(1,3)(2,4)} \\
\byb{(1,4,5,3,2)}\bya{(2,5,3)}\bya{(1,5,3,2,4)}\bybi{(1,3,4,2,5)} \\
\bya{(1,5)(3,4)}\bya{(1,2,3,4,5)}\byb{(1,3,5,2,4)}\byai{(2,3,5)} \\
\byai{(1,4,3,5,2)}\bybi{(1,2,4,5,3)}\bya{(1,4,2,5,3)}\bya{(1,5,3)} \\
\byb{(1,2)(3,4)}\bya{(2,3,4)}\bya{(1,3,4)}\bybi{(1,5)(2,3)} \\
\byai{(1,4,3,2,5)}\byai{(1,3,2,4,5)}\byb{(1,4)(3,5)}\byai{(2,4)(3,5)} \\
\byai{(1,2)(3,5)}\bybi{(1,3)(4,5)}\bybi{(1,4)(2,3)}\byai{(2,4,3)} \\
\byai{(1,3,2)}\bybi{(1,5,4,2,3)}\byai{(1,2,5,4,3)}\byb{(1,5,2)} \\
\bya{(2,4,5)}\bya{(1,4)(2,5)}\bybi{(1,2,4,3,5)}\byai{(1,3,5)} \\
\byai{(1,4,2,3,5)}\byb{(1,3,2,5,4)}\byai{(2,3)(4,5)}\byai{(1,5,4,3,2)} \\
\byb{e}\end{array}
\end{align*}
\end{aid}

\begin{aid} \label{(124)+(12345)C2}
A second hamiltonian cycle~$C_2$ in $\Cay(A_5; \quot a, \quot b)$ with $\quot a = (1,2,4)$ and $\quot b = (1,2,3,4,5)$.
\begin{align*}
\begin{array}{ccccccccccc}
&&e\bya{(1,2,4)}\bya{(1,4,2)}\bybi{(1,5,2,4,3)} \\
\bybi{(1,2,5,3,4)}\byai{(3,4,5)}\byai{(1,5,3,4,2)}\bybi{(1,3)(2,5)} \\
\byai{(1,4,5,2,3)}\byb{(1,3,5,4,2)}\bya{(3,5,4)}\bya{(1,2,3,5,4)} \\
\bybi{(1,4,5)}\bya{(1,2,5)}\bya{(1,5)(2,4)}\bybi{(2,5)(3,4)} \\
\bya{(1,5,2,3,4)}\bya{(1,3,4,5,2)}\bybi{(1,2,3)}\bybi{(1,5,4)} \\
\byai{(2,5,4)}\byai{(1,2)(4,5)}\bybi{(1,4,3)}\byai{(1,3)(2,4)} \\
\byb{(1,4,5,3,2)}\bya{(2,5,3)}\bya{(1,5,3,2,4)}\bybi{(1,3,4,2,5)} \\
\bya{(1,5)(3,4)}\bya{(1,2,3,4,5)}\byb{(1,3,5,2,4)}\byai{(2,3,5)} \\
\byai{(1,4,3,5,2)}\bybi{(1,2,4,5,3)}\bya{(1,4,2,5,3)}\bya{(1,5,3)} \\
\byb{(1,2)(3,4)}\bya{(2,3,4)}\bya{(1,3,4)}\bybi{(1,5)(2,3)} \\
\byai{(1,4,3,2,5)}\byai{(1,3,2,4,5)}\byb{(1,4)(3,5)}\byai{(2,4)(3,5)} \\
\byai{(1,2)(3,5)}\bybi{(1,3)(4,5)}\bybi{(1,4)(2,3)}\byai{(2,4,3)} \\
\byai{(1,3,2)}\bybi{(1,5,4,2,3)}\byai{(1,2,5,4,3)}\byb{(1,5,2)} \\
\bya{(2,4,5)}\bya{(1,4)(2,5)}\bybi{(1,2,4,3,5)}\byai{(1,3,5)} \\
\byai{(1,4,2,3,5)}\byb{(1,3,2,5,4)}\byai{(2,3)(4,5)}\byai{(1,5,4,3,2)} \\
\byb{e}
\end{array}
\end{align*}
\end{aid}

\begin{aid} \label{(12345)+(12354)}
A hamiltonian cycle~$C_1$ in $\Cay(A_5; \quot a, \quot b)$ with $\quot a = (1,2,3,4,5)$ and $\quot b = (1,2,3,5,4)$.
\begin{align*}
\begin{array}{ccccccccccc}
&&e\byb{(1,2,3,5,4)}\byai{(1,4,5)}\byb{(1,2,3)} \\
\bya{(1,3,4,5,2)}\byb{(2,4,3)}\byb{(1,4)(3,5)}\byb{(1,2,5)} \\
\byb{(1,5,4,2,3)}\bya{(1,3,2)}\byb{(3,5,4)}\byb{(1,2,5,3,4)} \\
\byai{(1,3,5)}\byb{(1,2,5,4,3)}\byb{(1,5,3,4,2)}\byai{(1,3)(2,5)} \\
\bybi{(1,4,2,3,5)}\byai{(2,4,5)}\byb{(1,4)(2,3)}\byai{(1,5)(2,4)} \\
\byb{(1,4,5,2,3)}\bya{(1,3,5,4,2)}\byb{(2,5)(3,4)}\byb{(1,5,3,2,4)} \\
\byai{(1,3,4,2,5)}\byb{(1,5,2,4,3)}\byb{(1,4,5,3,2)}\byai{(1,3)(2,4)} \\
\bybi{(1,2,3,4,5)}\bya{(1,3,5,2,4)}\bybi{(2,3,4)}\bybi{(1,2)(4,5)} \\
\byai{(1,4,3)}\bybi{(1,3,2,4,5)}\bybi{(1,5,2,3,4)}\bybi{(2,5,4)} \\
\bybi{(1,2)(3,5)}\byai{(1,3)(4,5)}\bybi{(1,5)(2,3)}\bya{(1,3,4)} \\
\bybi{(2,3)(4,5)}\bybi{(1,5,2)}\bybi{(1,4,2,5,3)}\bybi{(1,2,4,3,5)} \\
\bya{(1,4)(2,5)}\bybi{(2,4)(3,5)}\bybi{(1,2)(3,4)}\byai{(1,5,3)} \\
\bybi{(1,4,3,2,5)}\bya{(1,5,4)}\bybi{(2,5,3)}\bybi{(1,4,3,5,2)} \\
\byai{(1,2,4,5,3)}\byb{(1,4,2)}\byb{(2,3,5)}\byb{(1,3,2,5,4)} \\
\byb{(1,5)(3,4)}\bya{(1,2,4)}\bybi{(3,4,5)}\bybi{(1,5,4,3,2)} \\
\bya{e}
\end{array}
\end{align*}
\end{aid}

\begin{aid} \label{(12345)+(13425)C1}
A hamiltonian cycle~$C_1$ in $\Cay(A_5; \quot a, \quot b)$ with $\quot a = (1,2,3,4,5)$ and $\quot b = (1,3,4,2,5)$.
\begin{align*}
\begin{array}{ccccccccccc}
&&e\bya{(1,2,3,4,5)}\bya{(1,3,5,2,4)}\bya{(1,4,2,5,3)} \\
\bya{(1,5,4,3,2)}\bybi{(1,4,2,3,5)}\byai{(2,4,5)}\byai{(1,2)(3,4)} \\
\byai{(1,5,3)}\byai{(1,3,2,5,4)}\byb{(1,2,4,5,3)}\bya{(1,4,3,5,2)} \\
\bya{(2,5,4)}\bya{(1,5)(2,3)}\bya{(1,3,4)}\bybi{(1,5,2)} \\
\bya{(2,3,4)}\bya{(1,3,2,4,5)}\byb{(1,2)(3,5)}\byai{(1,3)(4,5)} \\
\byai{(1,4)(2,3)}\byb{(1,2,5,4,3)}\byai{(1,4)(3,5)}\byb{(1,5,4,2,3)} \\
\byai{(1,4)(2,5)}\byai{(1,2,4,3,5)}\byai{(3,4,5)}\byai{(1,3,2)} \\
\bybi{(1,5)(2,4)}\byai{(2,5)(3,4)}\bybi{(1,2,3)}\bya{(1,3,4,5,2)} \\
\bya{(2,4)(3,5)}\bya{(1,4,3,2,5)}\bya{(1,5,4)}\byb{(1,3)(2,4)} \\
\byai{(1,5,2,3,4)}\byai{(1,2,5)}\byai{(3,5,4)}\byai{(1,4,5,3,2)} \\
\bybi{(1,3,4,2,5)}\byai{(2,3,5)}\byai{(1,2)(4,5)}\byai{(1,4,3)} \\
\byai{(1,5,3,2,4)}\bybi{(1,3,5,4,2)}\bya{(2,5,3)}\bya{(1,5)(3,4)} \\
\byb{(1,4,2)}\bya{(2,3)(4,5)}\bya{(1,3,5)}\byb{(1,5,3,4,2)} \\
\bya{(2,4,3)}\bya{(1,4,5)}\bya{(1,2,3,5,4)}\bya{(1,3)(2,5)} \\
\bybi{(1,2,4)}\bya{(1,4,5,2,3)}\bybi{(1,2,5,3,4)}\bya{(1,5,2,4,3)} \\
\byb{e}
\end{array}
\end{align*}
\end{aid}

\begin{aid} \label{(12345)+(13425)C2}
A second hamiltonian cycle~$C_2$ in $\Cay(A_5; \quot a, \quot b)$ with $\quot a = (1,2,3,4,5)$ and $\quot b = (1,3,4,2,5)$.
\begin{align*}
\begin{array}{ccccccccccc}
&&e\bya{(1,2,3,4,5)}\bya{(1,3,5,2,4)}\bya{(1,4,2,5,3)} \\
\bya{(1,5,4,3,2)}\bybi{(1,4,2,3,5)}\byai{(2,4,5)}\byai{(1,2)(3,4)} \\
\byai{(1,5,3)}\byai{(1,3,2,5,4)}\byb{(1,2,4,5,3)}\bya{(1,4,3,5,2)} \\
\bya{(2,5,4)}\bya{(1,5)(2,3)}\bya{(1,3,4)}\bybi{(1,5,2)} \\
\bya{(2,3,4)}\bya{(1,3,2,4,5)}\byb{(1,2)(3,5)}\byai{(1,3)(4,5)} \\
\bybi{(1,4)(2,5)}\byai{(1,2,4,3,5)}\byai{(3,4,5)}\byai{(1,3,2)} \\
\byai{(1,5,4,2,3)}\bybi{(1,4)(3,5)}\bya{(1,2,5,4,3)}\bybi{(1,4)(2,3)} \\
\byai{(1,5)(2,4)}\byai{(2,5)(3,4)}\bybi{(1,2,3)}\bya{(1,3,4,5,2)} \\
\bya{(2,4)(3,5)}\bya{(1,4,3,2,5)}\bya{(1,5,4)}\byb{(1,3)(2,4)} \\
\byai{(1,5,2,3,4)}\byai{(1,2,5)}\byai{(3,5,4)}\byai{(1,4,5,3,2)} \\
\bybi{(1,3,4,2,5)}\byai{(2,3,5)}\byai{(1,2)(4,5)}\byai{(1,4,3)} \\
\byai{(1,5,3,2,4)}\bybi{(1,3,5,4,2)}\bya{(2,5,3)}\bya{(1,5)(3,4)} \\
\byb{(1,4,2)}\bya{(2,3)(4,5)}\bya{(1,3,5)}\byb{(1,5,3,4,2)} \\
\bya{(2,4,3)}\bya{(1,4,5)}\bya{(1,2,3,5,4)}\bya{(1,3)(2,5)} \\
\bybi{(1,2,4)}\bya{(1,4,5,2,3)}\bybi{(1,2,5,3,4)}\bya{(1,5,2,4,3)} \\
\byb{e}
\end{array}
\end{align*}
\end{aid}

\begin{aid} \label{S<4}
Write $\quot S = \{s_1,\ldots,s_r\}$ and let $H_i = \langle s_1,\ldots,s_i \rangle$. Suppose $r \ge 4$. Since $\quot S$ is a minimal generating set, we have $H_{i-1} \subsetneq H_i$ for each~$i$. Since $|A_5| = 2^2 \cdot 3 \cdot 5$ is the product of only $4$~primes, we must have $r = 4$ and $|H_i : H_{i-1}|$ is prime for $i = 1,\ldots,4$. 
(Recall that $A_5$ is simple, so the natural action of~$A_5$ on the coset space $A_5/H$ of any proper subgroup~$H$ must be faithful. This implies $|A_5 : H| > 4$, because $|A_5| > |S_4|$.)
Since $A_4$ is the only subgroup of prime index in~$A_5$ (up to conjugacy), we may assume $H_3 = A_4$. Then $H_2$ must be the Sylow $2$-subgroup of~$A_4$, since that is the only subgroup of prime index. So $H_3 = A_4$ is generated by $s_1$ and~$s_3$, contradicting the minimality of~$\quot S$. 
\end{aid}

\begin{aid} \label{Why(125)(135)(145)}
The minimality of~$\quot S$ implies that $\langle \quot a, \quot b \rangle \neq A_5$, so there must be at least two elements in the intersection of the supports of $\quot a$ and~$\quot b$. The supports cannot be equal, since $\quot a \neq \quot b\,^{\pm1}$. Therefore, the intersection of the support consists of two points, which we may assume are $1$ and~$5$. Then we may assume $\quot a = (1,2,5)$ and $\quot b = (1,3,5)$. Now the support of~$\quot c$ must contain exactly two points from the support of~$\quot a$ and exactly two points from the support of~$\quot b$. This implies that the support of~$\quot c$ is $\{1,4,5\}$. so $\quot c = (1,4,5)^{\pm1}$.
\end{aid}

\begin{aid} \label{(125)+(135)+(145)}
A hamiltonian cycle~$R_1$ in $\Cay(A_5; \quot a, \quot b, \quot c)$ with $\quot a = (1,2,5)$, $\quot b = (1,3,5)$, and $\quot c = (1,4,5)$.
\begin{align*}
\begin{array}{ccccccccccc}
&&e\bya{(1,2,5)}\bya{(1,5,2)}\byb{(1,3,2)} \\
\bya{(2,5,3)}\bya{(1,5)(2,3)}\byb{(1,2,3)}\bya{(1,3)(2,5)} \\
\bya{(1,5,3)}\byc{(1,4,3)}\bya{(1,2,5,4,3)}\bya{(1,5,2,4,3)} \\
\byb{(2,4,3)}\bya{(1,4,3,2,5)}\bya{(1,5,4,3,2)}\byb{(1,2)(3,4)} \\
\bya{(2,5)(3,4)}\bya{(1,5)(3,4)}\byc{(1,3,4)}\bya{(1,2,5,3,4)} \\
\byb{(1,4)(2,5)}\byb{(1,3,2,5,4)}\bya{(1,5,3,2,4)}\bya{(1,4)(2,3)} \\
\byb{(1,2,3,5,4)}\bya{(1,3,5,2,4)}\bya{(1,4)(3,5)}\byc{(3,5,4)} \\
\byc{(1,3,5)}\bya{(1,2)(3,5)}\bya{(2,3,5)}\byc{(1,4,2,3,5)} \\
\byb{(1,5,4,2,3)}\bya{(1,3)(2,4)}\byb{(2,4)(3,5)}\byc{(1,2,4,3,5)} \\
\bya{(1,4,3,5,2)}\byc{(1,3,5,4,2)}\byb{(1,5,3,4,2)}\byb{(1,4,2)} \\
\byc{(1,2)(4,5)}\byb{(1,3,4,5,2)}\byb{(1,4,5,3,2)}\bya{(2,3)(4,5)} \\
\bya{(1,3,2,4,5)}\byb{(1,2,4,5,3)}\bya{(1,4,5,2,3)}\bya{(1,3)(4,5)} \\
\byb{(3,4,5)}\bya{(1,2,3,4,5)}\byc{(1,5,2,3,4)}\byc{(2,3,4)} \\
\bya{(1,3,4,2,5)}\byb{(1,4,2,5,3)}\byb{(2,5,4)}\bya{(1,5)(2,4)} \\
\byc{(1,2,4)}\byc{(2,4,5)}\bya{(1,4,5)}\byc{(1,5,4)} \\
\byc{e}
\end{array}
\end{align*}
\end{aid}

\begin{aid} \label{Why(12)(45)+(123)+(124)}
By renumbering, we may assume $\quot b = (1,2,3)$. Then, since $\langle \quot b, \quot c \rangle = A_4$, we know that $4$ is in the support of~$\quot c$, and that the other two elements of the support are in $\{1,2,3\}$. Therefore, after conjugating by a power of~$\quot b$, we may assume the support of~$\quot c$ is $\{1,2,4\}$. So $\quot c = (1,2,4)^{\pm1}$.

Since $\langle \quot a, \quot b, \quot c \rangle = A_5$ is transitive on $\{1,2,3,4,5\}$, we know that $5$ is in the support of~$\quot a$.
Since $\langle \quot a, \quot b \rangle \neq A_5$, we know that the support of~$\quot b$ does not contain precisely one element of each of the cycles of~$\quot a$ (and similarly for the support of~$\quot c$).

If one of the cycles of~$\quot a$ is disjoint from the support of~$\quot b$, then the cycle must be  $(4,5)$. The support of~$\quot c$ contains precisely one element of this cycle, and cannot be disjoint from the other cycle, so it must contain the entire cycle. The only $2$-element subset of $\{1,2,3\}$ contained in the support of~$\quot c$ is $\{1,2\}$, so $\quot a = (1,2)(3,4)$.

We may now assume that no cycle of~$\quot a$ is disjoint from the support of~$\quot b$ or the support of~$\quot c$. So the cycle $(x,5)$ in $\quot a$ must be either $(1,5)$ or $(2,5)$. We may assume it is $(1,5)$ (after conjugating by $(1,2)$ and replacing $\quot b$ and $\quot c$ by their inverses, if necessary). The other cycle either is $(2,3)$ (in which case, $\langle \quot a, \quot c \rangle = A_5$), or is either $(2,4)$ or $(3,4)$ (in which case, $\langle \quot a, \quot b \rangle = A_5$), which contradicts the minimality of $\quot S$.
\end{aid}

\begin{aid} \label{(12)(45)+(123)+(124)C1}
A hamiltonian cycle~$C_1$ in $\Cay(A_5; \quot a, \quot b, \quot c, \quot c)$ with $\quot a = (1,2)(4,5)$, $\quot b = (1,2,3)$, and $\quot c = (1,2,4)$.
\begin{align*}
\begin{array}{ccccccccccc}
&&e\bya{(1,2)(4,5)}\byci{(1,5,4)}\bya{(1,2,5)} \\
\byb{(1,5)(2,3)}\bya{(1,3,2,5,4)}\byc{(1,5,4,3,2)}\bya{(2,5,3)} \\
\byb{(1,5,3)}\byb{(1,2)(3,5)}\bya{(3,5,4)}\byb{(1,2,5,4,3)} \\
\byc{(1,5,4,2,3)}\bybi{(2,5,4)}\bya{(1,5,2)}\bybi{(1,3)(2,5)} \\
\bybi{(2,3,5)}\bya{(1,3,5,4,2)}\byci{(1,2,3,5,4)}\bya{(1,3,5)} \\
\byci{(1,4,2,3,5)}\byb{(1,3,4,2,5)}\byb{(1,5)(2,4)}\byc{(1,4,5)} \\
\bybi{(1,3,2,4,5)}\bybi{(1,2,3,4,5)}\bya{(1,3,4)}\bybi{(1,4)(2,3)} \\
\bybi{(1,2,4)}\byc{(1,4,2)}\bya{(2,4,5)}\byb{(1,4,5,2,3)} \\
\byci{(1,5,2,4,3)}\bybi{(2,5)(3,4)}\bya{(1,5,3,4,2)}\byci{(1,2,5,3,4)} \\
\bya{(1,5)(3,4)}\bybi{(1,4,3,2,5)}\bybi{(1,2,4,3,5)}\bya{(1,4)(3,5)} \\
\bybi{(1,5,3,2,4)}\byc{(1,4,5,3,2)}\bya{(2,4,3)}\byb{(1,4,3)} \\
\byb{(1,2)(3,4)}\bya{(3,4,5)}\byb{(1,2,4,5,3)}\byc{(1,4,2,5,3)} \\
\bybi{(2,4)(3,5)}\bya{(1,4,3,5,2)}\byci{(1,3,5,2,4)}\byb{(1,4)(2,5)} \\
\byb{(1,5,2,3,4)}\byc{(1,3,4,5,2)}\bya{(2,3,4)}\byb{(1,3)(2,4)} \\
\byci{(1,2,3)}\bya{(1,3)(4,5)}\bybi{(2,3)(4,5)}\bya{(1,3,2)} \\
\byb{e}
\end{array}
\end{align*}
\end{aid}

\begin{aid} \label{(12)(45)+(123)+(124)C2}
A second hamiltonian cycle~$C_2$ in $\Cay(A_5; \quot a, \quot b, \quot c)$ with $\quot a = (1,2)(4,5)$, $\quot b = (1,2,3)$, and $\quot c = (1,2,4)$.
\begin{align*}
\begin{array}{ccccccccccc}
&&e\bya{(1,2)(4,5)}\byci{(1,5,4)}\bya{(1,2,5)} \\
\byb{(1,5)(2,3)}\bya{(1,3,2,5,4)}\byc{(1,5,4,3,2)}\bya{(2,5,3)} \\
\byb{(1,5,3)}\byb{(1,2)(3,5)}\bya{(3,5,4)}\byb{(1,2,5,4,3)} \\
\byc{(1,5,4,2,3)}\bybi{(2,5,4)}\bya{(1,5,2)}\byci{(1,4)(2,5)} \\
\byb{(1,5,2,3,4)}\byb{(1,3,5,2,4)}\byc{(1,4,3,5,2)}\bya{(2,4)(3,5)} \\
\byb{(1,4,2,5,3)}\byci{(1,2,4,5,3)}\bybi{(3,4,5)}\bya{(1,2)(3,4)} \\
\bybi{(1,4,3)}\bybi{(2,4,3)}\bya{(1,4,5,3,2)}\byci{(1,5,3,2,4)} \\
\byb{(1,4)(3,5)}\bya{(1,2,4,3,5)}\byb{(1,4,3,2,5)}\byb{(1,5)(3,4)} \\
\bya{(1,2,5,3,4)}\byc{(1,5,3,4,2)}\bya{(2,5)(3,4)}\byb{(1,5,2,4,3)} \\
\byci{(1,3)(2,5)}\bybi{(2,3,5)}\bya{(1,3,5,4,2)}\byci{(1,2,3,5,4)} \\
\bya{(1,3,5)}\byci{(1,4,2,3,5)}\byb{(1,3,4,2,5)}\byb{(1,5)(2,4)} \\
\byc{(1,4,5)}\bybi{(1,3,2,4,5)}\bybi{(1,2,3,4,5)}\bya{(1,3,4)} \\
\bybi{(1,4)(2,3)}\bybi{(1,2,4)}\byc{(1,4,2)}\bya{(2,4,5)} \\
\byb{(1,4,5,2,3)}\byb{(1,3,4,5,2)}\bya{(2,3,4)}\byb{(1,3)(2,4)} \\
\byci{(1,2,3)}\bya{(1,3)(4,5)}\bybi{(2,3)(4,5)}\bya{(1,3,2)} \\
\byb{e}
\end{array}
\end{align*}
\end{aid}

\begin{aid} \label{Why(12)(45)+(12)(34)+(124)}
By assumption, one of the $2$-cycles in~$\quot a$ is $(4,5)$. The other $2$-cycle must be contained in $\{1,2,3\}$, so, after conjugating by a power of~$\quot c$, we may assume $\quot a = (1,2)(4,5)$.

Since $\langle \quot a, \quot b, \quot c \rangle$ is transitive on $\{1,2,3,4,5\}$, the permutation~$\quot b$ cannot have $(4,5)$ as one of its $2$-cycles. So no $2$-cycle in $\quot b$ is disjoint from the support of~$\quot c$. So one of the $2$-cycles must be contained in the support of~$\quot c$, which is $\{1,2,3\}$. This implies that either $4$ or~$5$ must be fixed by~$\quot b$. We may assume it is $5$ that is fixed (after conjugating by $(4,5)$ if necessary). So $\quot b$ is either $(1,2)(3,4)$ or $(1,3)(2,4)$ or $(2,3)(1,4)$. Since the last two are conjugate by $(1,2)$ (which centralizes $\quot a$ and inverts~$\quot b$), they do not both need to be considered.
\end{aid}

\begin{aid} \label{(12)(45)+(12)(34)+(123)}
A hamiltonian cycle in $\Cay(A_5; \quot a, \quot b, \quot c)$ with $\quot a = (1,2)(4,5)$, $\quot b = (1,2)(3,4)$, and $\quot c = (1,2,3)$.
\begin{align*}
\begin{array}{ccccccccccc}
&&e\bya{(1,2)(4,5)}\byci{(1,3)(4,5)}\bya{(1,2,3)} \\
\byb{(1,3,4)}\bya{(1,2,3,4,5)}\byb{(1,3,5)}\bya{(1,2,3,5,4)} \\
\byci{(1,5,4)}\bya{(1,2,5)}\byb{(1,5)(3,4)}\byci{(1,4,3,2,5)} \\
\byb{(1,5)(2,4)}\byc{(1,4,2,3,5)}\byb{(1,3,2,4,5)}\bya{(1,4)(2,3)} \\
\byb{(1,3)(2,4)}\bya{(1,4,5,2,3)}\byc{(1,3,4,5,2)}\bya{(2,3,4)} \\
\byb{(1,3,2)}\bya{(2,3)(4,5)}\byb{(1,3,5,4,2)}\bya{(2,3,5)} \\
\byc{(1,3)(2,5)}\bya{(1,5,4,2,3)}\byb{(1,3,2,5,4)}\bya{(1,5)(2,3)} \\
\byb{(1,3,4,2,5)}\bya{(1,5,2,3,4)}\byc{(1,3,5,2,4)}\byc{(1,4)(2,5)} \\
\byb{(1,5,2,4,3)}\byci{(2,5)(3,4)}\byb{(1,5,2)}\bya{(2,5,4)} \\
\byb{(1,5,4,3,2)}\bya{(2,5,3)}\byb{(1,5,3,4,2)}\byci{(1,4,2,5,3)} \\
\byci{(2,4)(3,5)}\bya{(1,4,3,5,2)}\byb{(2,4,5)}\bya{(1,4,2)} \\
\byb{(2,4,3)}\bya{(1,4,5,3,2)}\byci{(1,2,4,5,3)}\bya{(1,4,3)} \\
\byb{(1,2,4)}\bya{(1,4,5)}\byb{(1,2,4,3,5)}\bya{(1,4)(3,5)} \\
\byci{(1,5,3,2,4)}\byci{(1,2,5,3,4)}\byb{(1,5,3)}\bya{(1,2,5,4,3)} \\
\byci{(3,5,4)}\bya{(1,2)(3,5)}\byb{(3,4,5)}\bya{(1,2)(3,4)} \\
\byb{e}
\end{array}
\end{align*}
\end{aid}

\begin{aid} \label{(12)(45)+(13)(24)+(123)}
A hamiltonian cycle in $\Cay(A_5; \quot a, \quot b, \quot c)$ with $\quot a = (1,2)(4,5)$, $\quot b = (1,3)(2,4)$, and $\quot c = (1,2,3)$.
\begin{align*}
\begin{array}{ccccccccccc}
&&e\bya{(1,2)(4,5)}\byb{(1,3,2,5,4)}\bya{(1,5)(2,3)} \\
\byb{(1,2,4,3,5)}\bya{(1,4)(3,5)}\byb{(1,5,3,4,2)}\bya{(2,5)(3,4)} \\
\byb{(1,4,5,2,3)}\byc{(1,3,4,5,2)}\bya{(2,3,4)}\byb{(1,4,3)} \\
\bya{(1,2,4,5,3)}\byb{(2,5,3)}\bya{(1,5,4,3,2)}\byb{(1,2,3,5,4)} \\
\bya{(1,3,5)}\byb{(1,5)(2,4)}\bya{(1,4)(2,5)}\byci{(1,3,5,2,4)} \\
\byb{(1,5,2)}\bya{(2,5,4)}\byb{(1,3)(4,5)}\byci{(2,3)(4,5)} \\
\byb{(1,2,5,4,3)}\bya{(1,5,3)}\byb{(2,4)(3,5)}\bya{(1,4,3,5,2)} \\
\byb{(1,5,2,3,4)}\bya{(1,3,4,2,5)}\byci{(1,4,2,3,5)}\byb{(1,5)(3,4)} \\
\bya{(1,2,5,3,4)}\byb{(1,4,5,3,2)}\bya{(2,4,3)}\byb{(1,2,3)} \\
\byc{(1,3,2)}\byb{(1,2,4)}\bya{(1,4,5)}\byci{(1,3,2,4,5)} \\
\byb{(1,2,5)}\bya{(1,5,4)}\byb{(1,3,5,4,2)}\bya{(2,3,5)} \\
\byb{(1,5,2,4,3)}\bya{(1,4,2,5,3)}\byb{(3,4,5)}\bya{(1,2)(3,4)} \\
\byb{(1,4)(2,3)}\byc{(1,3,4)}\bya{(1,2,3,4,5)}\byb{(1,4,3,2,5)} \\
\bya{(1,5,3,2,4)}\byb{(1,2)(3,5)}\bya{(3,5,4)}\byb{(1,5,4,2,3)} \\
\bya{(1,3)(2,5)}\byb{(2,4,5)}\bya{(1,4,2)}\byci{(1,3)(2,4)} \\
\byb{e}
\end{array}
\end{align*}
\end{aid}

\begin{aid} \label{Why(12)(34)+(12)(35)+(123)}
Since $\langle \quot a, \quot c \rangle \neq A_5$, and $\quot a$ does not interchange $4$ and~$5$, we see that $\quot a$ must fix either $4$ or~$5$. Similarly for~$\quot b$. Furthermore, $\quot a$ and~$\quot b$ cannot both fix~$4$ (or~$5$), since $\langle \quot a, \quot b, \quot c \rangle$ is transitive. So one of them fixes~$4$, and the other fixes~$5$.

We may assume it is $\quot a$ that fixes~$5$. Then we may assume $\quot a = (1,2)(3,4)$, after conjugating by a power of~$\quot c$.

Then $\quot b$ must be either $(1,2)(3,5)$ or $(1,3)(2,5)$ or $(1,5)(2,3)$. However, the last two are conjugate by $(1,2)$ (which centralizes~$\quot a$ and inverts $\quot c$), so they do not both need to be considered.
\end{aid}

\begin{aid} \label{(12)(34)+(12)(35)+(123)}
A hamiltonian cycle in $\Cay(A_5; \quot a, \quot b, \quot c)$ with $\quot a = (1,2)(3,4)$, $\quot b = (1,2)(3,5)$, and $\quot c = (1,2,3)$.
\begin{align*}
\begin{array}{ccccccccccc}
&&e\bya{(1,2)(3,4)}\byci{(1,4,3)}\bya{(1,2,4)} \\
\byc{(1,4)(2,3)}\byb{(1,3,5,2,4)}\bya{(1,4,5,2,3)}\byb{(1,3,2,4,5)} \\
\bya{(1,4,2,3,5)}\byb{(1,3)(2,4)}\byc{(1,4,2)}\bya{(2,4,3)} \\
\byb{(1,4,3,5,2)}\bya{(2,4,5)}\byb{(1,4,5,3,2)}\bya{(2,4)(3,5)} \\
\byc{(1,4,2,5,3)}\bya{(1,5,3,2,4)}\byc{(1,4)(3,5)}\bya{(1,2,4,5,3)} \\
\byb{(1,4,5)}\bya{(1,2,4,3,5)}\byci{(1,5)(3,4)}\byci{(1,4,3,2,5)} \\
\bya{(1,5)(2,4)}\byci{(1,3,4,2,5)}\byb{(1,5,4,2,3)}\bya{(1,3,2,5,4)} \\
\byc{(1,5,4)}\byb{(1,2,5,3,4)}\bya{(1,5,3)}\byb{(1,2,5)} \\
\byc{(1,5)(2,3)}\byc{(1,3,5)}\bya{(1,2,3,4,5)}\byb{(1,3)(4,5)} \\
\bya{(1,2,3,5,4)}\byb{(1,3,4)}\bya{(1,2,3)}\byc{(1,3,2)} \\
\bya{(2,3,4)}\byb{(1,3,5,4,2)}\bya{(2,3)(4,5)}\byb{(1,3,4,5,2)} \\
\bya{(2,3,5)}\byc{(1,3)(2,5)}\bya{(1,5,2,3,4)}\byci{(1,4)(2,5)} \\
\bya{(1,5,2,4,3)}\byci{(2,5)(3,4)}\bya{(1,5,2)}\byb{(2,5,3)} \\
\bya{(1,5,3,4,2)}\byb{(2,5,4)}\bya{(1,5,4,3,2)}\byci{(1,2,5,4,3)} \\
\byci{(3,5,4)}\bya{(1,2)(4,5)}\byb{(3,4,5)}\bya{(1,2)(3,5)} \\
\byb{e}
\end{array}
\end{align*}
\end{aid}

\begin{aid} \label{(12)(34)+(13)(25)+(123)}
A hamiltonian cycle in $\Cay(A_5; \quot a, \quot b, \quot c)$ with $\quot a = (1,2)(3,4)$, $\quot b = (1,3)(2,5)$, and $\quot c = (1,2,3)$.
\begin{align*}
\begin{array}{ccccccccccc}
&&e\bya{(1,2)(3,4)}\byb{(1,4,3,2,5)}\byci{(1,2,4,3,5)} \\
\bya{(1,4,5)}\byb{(1,3,4,5,2)}\bya{(2,3,5)}\byb{(1,5,3)} \\
\bya{(1,2,5,3,4)}\byc{(1,5,3,2,4)}\byb{(1,2,3,5,4)}\bya{(1,3)(4,5)} \\
\byb{(2,4,5)}\bya{(1,4,3,5,2)}\byb{(1,5)(3,4)}\bya{(1,2,5)} \\
\byb{(1,3,2)}\bya{(2,3,4)}\byb{(1,4,2,5,3)}\byc{(1,5,3,4,2)} \\
\bya{(2,5,3)}\byb{(1,2,3)}\bya{(1,3,4)}\byci{(1,4)(2,3)} \\
\bya{(1,3)(2,4)}\byb{(2,5,4)}\bya{(1,5,4,3,2)}\byci{(1,2,5,4,3)} \\
\byb{(2,4,3)}\bya{(1,4,2)}\byb{(1,3,4,2,5)}\byc{(1,5)(2,4)} \\
\byc{(1,4,2,3,5)}\bya{(1,3,2,4,5)}\byb{(1,2)(4,5)}\bya{(3,5,4)} \\
\byb{(1,5,2,4,3)}\bya{(1,4)(2,5)}\byc{(1,5,2,3,4)}\byb{(1,4)(3,5)} \\
\bya{(1,2,4,5,3)}\byb{(2,3)(4,5)}\bya{(1,3,5,4,2)}\byci{(1,5,4,2,3)} \\
\byb{(2,4)(3,5)}\bya{(1,4,5,3,2)}\byb{(1,2,3,4,5)}\bya{(1,3,5)} \\
\byci{(1,5)(2,3)}\byb{(1,2)(3,5)}\bya{(3,4,5)}\byb{(1,4,5,2,3)} \\
\bya{(1,3,5,2,4)}\byb{(1,5,4)}\byci{(1,3,2,5,4)}\byb{(1,2,4)} \\
\bya{(1,4,3)}\byb{(2,5)(3,4)}\bya{(1,5,2)}\byci{(1,3)(2,5)} \\
\byb{e}
\end{array}
\end{align*}
\end{aid}

\end{appendix}

\end{document}